\documentclass[reqno,a4paper, 11pt]{amsart}

\usepackage[a4paper=true,pdfpagelabels]{hyperref}
\usepackage{graphicx}

\usepackage[ansinew]{inputenc}
\usepackage{amsfonts,epsfig}
\usepackage{latexsym}
\usepackage{amsmath}
\usepackage{amssymb}
\usepackage{mathabx}
\usepackage{comment}

\newcommand{\C}{\mathbb{C}}
\newcommand{\N}{\mathbb{N}}
\newcommand{\D}{\mathbb{D}}

\newcommand{\DD}{\widehat{\mathcal{D}}}
\newcommand{\Dd}{\widecheck{\mathcal{D}}}
\newcommand{\DDD}{\mathcal{D}}

\newcommand{\Hcal}{\mathcal{H}}

\newcommand{\Zcal}{\mathcal{Z}}

\newcommand{\re}{\textrm{Re}\,}

\newcommand{\omh}{\widehat{\om}}

\def\a{\alpha}       \def\b{\beta}        \def\g{\gamma}
\def\d{\delta}           \def\e{\varepsilon}
     \def\om{\omega}      
\def\s{\sigma}       \def\t{\theta}       
         \def\r{\rho}         \def\z{\zeta}
                  \def\vp{\varphi}
\def\G{\Gamma}

\theoremstyle{definition}

\theoremstyle{plain}
\newtheorem{lem}{Lemma}
\newtheorem{theo}[lem]{Theorem}
\newtheorem{prop}[lem]{Proposition}
\newtheorem{coro}[lem]{Corollary}
\theoremstyle{definition}

\theoremstyle{plain}
\newtheorem{lettertheorem}{Theorem}

\newtheorem{letterlemma}[lettertheorem]{Lemma}

\newenvironment{pro}
{\noindent{\it Proof. }}{\hfill $\Box$\par\vspace{2.5mm}}

\newenvironment{proofof}[1]
{\noindent{\it Proof of #1. }}{\hfill $\Box$\par\vspace{2.5mm}}

\numberwithin{equation}{section}

\begin{document}

\title[Zero sequences, factorization and sampling]{Zero sequences, factorization and sampling measures for weighted Bergman spaces}


\thanks{This research was supported in part by Ministerio de Econom\'{\i}a y Competitivivad,
Spain, project MTM2014-52865-P; by Academy of Finland project no.~268009, and by Faculty of Science and Forestry of University of Eastern Finland.}


\date{\today}

\author{Taneli Korhonen}
\address{University of Eastern Finland, P.O.Box 111, 80101 Joensuu, Finland}
\email{taneli.korhonen@uef.fi}

\author{Jouni R\"atty\"a}
\address{University of Eastern Finland, P.O.Box 111, 80101 Joensuu, Finland}
\email{jouni.rattya@uef.fi}

\maketitle

\begin{abstract}
The zero sets of the Bergman space~$A^p_\om$ induced by either a radial weight~$\om$ admitting a certain doubling property or a non-radial Bekoll\'e-Bonami type weight are characterized in the spirit of Luecking's results from 1996. Accurate results obtained en route to this characterization are used to generalize Horowitz's factorization result from 1977 for functions in $A^p_\om$. The utility of the obtained factorization is illustrated by applications to integration and composition operators as well as to small Hankel operator induced by a conjugate analytic symbol. Dominating sets and sampling measures for the weighted Bergman space $A^p_\om$ induced by a doubling weight are also studied. Several open problems related to the scheme of the paper are posed.
\end{abstract}

\section{Introduction and main results}\label{sec:Intro}

Let $\Hcal(\D)$ denote the space of analytic functions in the unit disc $\D=\{z\in\C:|z|<1\}$ of the complex plane $\C$. A function $\omega:\D\to[0,\infty)$, integrable over $\D$, is called a
weight. It is radial if $\omega(z)=\omega(|z|)$ for all $z\in\D$. For
$0<p<\infty$ and a weight $\omega$, the weighted Bergman
space $A^p_\omega$ consists of $f\in\Hcal(\D)$ such that
    $$
    \|f\|_{A^p_\omega}^p=\int_\D|f(z)|^p\omega(z)\,dA(z)<\infty,
    $$
where $dA(z)=\frac{dx\,dy}{\pi}$ is the normalized
Lebesgue area measure on $\D$. As usual,~$A^p_\alpha$ stands for the classical weighted
Bergman space induced by the standard radial weight $\omega(z)=(1-|z|^2)^\alpha$, where
$-1<\alpha<\infty$. For $f\in\Hcal(\D)$ and $0<r<1$, set
    \begin{equation*}
    \begin{split}
    M_p(r,f)&=\left(\frac{1}{2\pi}\int_{0}^{2\pi} |f(re^{it})|^p\,dt\right)^{1/p},\quad
    0<p<\infty,
    \end{split}
    \end{equation*}
and $M_\infty(r,f)=\max_{|z|=r}|f(z)|$. For $0<p\le\infty$, the Hardy space $H^p$ consists of $f\in \Hcal(\mathbb D)$ such that $\|f\|_{H^p}=\sup_{0<r<1}M_p(r,f)<\infty$.

In this paper we are mainly interested in zero-sequences, factorization, dominating sets and sampling measures for the Bergman space $A^p_\om$ induced by either a non-radial weight belonging to a kind of Bekoll\'e-Bonami class or a radial weight admitting a certain doubling property. Our studies of zeros and factorization go hand-in-hand and dominating sets and sampling measures are almost equally interrelated. While our results on zeros and factorization improve and generalize certain results in the literature, our findings on dominating sets and sampling measures are less complete but offer a somewhat new approach to these topics and also give arise to several open problems.

We begin with zeros and factorization. Let $\om:\D\to[0,\infty)$ be a weight, $0 < p < \infty$ and $f \in A^p_\om$ such that $f\not\equiv0$. Then a sequence $Z \subset \D$ is called the \emph{zero set (or sequence)} of $f$, denoted by $\Zcal(f)$, if $f(a) = 0$ for all $a \in Z$, counting multiplicities, and $f(a) \neq 0$ for all $a\in\D\setminus Z$. A set $Z$ is called a \emph{zero set} for $A^p_\om$ if there exists a nonzero function $f \in A^p_\om$ such that $Z = \Zcal(f)$. The set of all zero sets of $A^p_\om$ is denoted by $\Zcal(A^p_\om)$.

Zeros of functions in Hardy spaces are neatly characterized by the Blaschke condition and each Hardy-function~$f$ admits the well-known inner-outer factorization, where the inner part is a product of a singular inner function and a Blaschke product containing all the zeros of~$f$ and the non-vanishing outer part has the same norm as~$f$~\cite{Duren1970}. The situation of Bergman spaces is completely different because the distribution of zeros of functions in Bergman spaces is not that well understood neither such an efficient factorization as in the Hardy space case is known. Even if the geometric distribution of zeros is not completely described, the difference between known sufficient and necessary conditions for a sequence to be a zero set for~$A^p_\alpha$ is small. Probably the most commonly known results in this context are due to Luecking, Korenblum, Hedenmalm, Horowitz and Seip. Horowitz~\cite{Horzeros,Horzeros1,Horzeros2} studied unions, subsets and dependence on~$p$ of the zero sets of functions in $A^p_\alpha$ and obtained accurate information on certain sums involving the moduli of zeros. Some of these results were generalized to certain~$A^p_\om$ in \cite{PR2014Memoirs}. The studies by Korenblum~\cite{Kor}, Hedenmalm~\cite{HedStPeter} and Seip~\cite{S1,S2} employ methods based on the use of densities defined in terms of partial Blaschke sums, Stolz star domains and Beurling-Carleson characteristic of the corresponding boundary set, and yield more complete results. Luecking~\cite{L1996} gave a description of zero sets of $A^p_\alpha$ in terms of certain auxiliary functions induced by the zeros. Even if this characterization does not reveal the geometric distribution very transparently, it yields accurate information on the subsets of zero sets. Horowitz~\cite{HorFacto} also established a useful factorization theorem for functions in $A^p_\alpha$, but this factorization does not allow to take one of factors non-vanishing, and thus does not behave equally well as the inner-outer factorization in Hardy spaces. This factorization result was generalized to some~$A^p_\om$ in~\cite{HorSchna,PR2014Memoirs}.

We start with employing Luecking's~\cite{L1996} approach to describe $A^p_\om$ zero sets when $\om$ belongs to a kind of Bekoll\'e-Bonami class. For $1<q<\infty$, write $\om\in B_q$ if the weight $\om$ is (almost everywhere) strictly positive and
    \begin{equation*}
    B_q(\om)=\sup_{S}\left(\frac{1}{|S|^2}\int_{S}\om(z)(1-|z|^2)^{2q}\,dA(z)\right)
        \left(\frac{1}{|S|^2}\int_{S} \om(z)^{-\frac1{q-1}}\,dA(z)\right)^{q-1}
        <\infty,
    \end{equation*}
where the supremum is taken over all Carleson squares $S\subset\D$, and $|S|$ denotes the Euclidean area of $S$. Denote $B_\infty = \bigcup_{q>1}B_q$ for short. Recall that each Carleson square is of the form
    $$
    S(z)=\left\{re^{i\theta}\in\D:|z|<r<1,\,|\arg ze^{-i\theta}|<\frac{1-|z|}2\right\},\quad z\in\D\setminus\{0\}.
    $$

In the beginning of Section~\ref{sec:ProofThm1} we briefly analyze the classes $B_q$. In particular, we show that $B_p\subsetneq B_q$ for $1<p<q<\infty$ and discuss their Kerman-Torchinsky properties. 
Following~Luecking~\cite{L1996}, for a sequence $Z \subset \D$, we use the notation
    \begin{equation*}
    \begin{split}
    \psi_Z(z)&= \prod_{a\in Z} \overline{a}\frac{a-z}{1-\overline{a}z} \exp\left(1-\overline{a}\frac{a-z}{1-\overline{a}z}\right), \quad z \in \D,\\
    W_Z(z)&= e^{k_Z(z)},\quad k_Z(z)= \frac{|z|^2}2 \sum_{a \in Z}\frac{\left(1-|a|^2\right)^2}{|1-\overline{a}z|^2}, \quad z \in \D.
    \end{split}
    \end{equation*}
The first of our main results on Bergman zero sets is a generalization of \cite[Theorem~3]{L1996} and reads as follows.

\begin{theo}\label{theo:zeroset}
Let $0 < p < \infty$, $\om \in B_\infty$ and $Z$ be a sequence in $\D$. Then the following statements are equivalent:
\begin{itemize}
    \item[\rm (a)] $Z \in \Zcal(A^p_\om)$;

    \item[\rm (b)] Any subsequence of $Z$ belongs to $\Zcal(A^p_\om)$;

    \item[\rm (c)] $\sum_{a \in Z} (1-|a|)^2 < \infty$ and there is a nowhere zero function $F \in \Hcal(\D)$ such that $FW_Z \in L^p_\om$;

    \item[\rm (d)] $\sum_{a \in Z} (1-|a|)^2 < \infty$ and there is a nonzero function $F \in \Hcal(\D)$ such that $FW_Z \in L^p_\om$.
\end{itemize}
Moreover, if \emph{(a)} is true, then the mapping $f \mapsto f/\psi_Z$ is a continuous isomorphism from $\left\{f \in A^p_\om : Z \subset \Zcal(f)\right\}$ onto $\left\{F \in \Hcal(\D) : FW_Z \in L^p_\om\right\}$.
\end{theo}

If $F$ is that of Theorem~\ref{theo:zeroset} (c), and  $h = -p\log|F|$, then $h$ is harmonic and
    $$
    |FW_Z|^p = |F|^pW_Z^p = e^{p\log|F|}e^{pk_Z} = \exp(pk_Z - h).
    $$
Therefore the equivalence of (a) and (c) in Theorem~\ref{theo:zeroset} leads to the following result.

\begin{coro}\label{theo:zeroset2}
Let $0 < p < \infty$, $\om \in B_\infty$ and $Z$ be a sequence in $\D$. Then $Z \in \Zcal(A^p_\om)$ if and only if there exists a harmonic function $h$ such that
    $$
    \int_\D \exp(p k_Z(z) - h(z)) \om(z)\,dA(z) < \infty.
    $$
\end{coro}

By performing a certain perturbation on a zero set, in this case moving the points closer to the boundary, it becomes a zero set for some other weighted Bergman space. A sequence $(z_n)$ in $\D$ is separated or equivalently uniformly discrete if $\inf_{k\ne n}\r_p(z_k,z_n)=\inf_{k\ne n}|\vp_{z_k}(z_n)|>0$, where $\varphi_a(z) = \frac{a-z}{1-\overline{a}z}$ is the standard automorphism of the unit disc.

\begin{coro}\label{coro:zeroPerturb}
Let $0<p<\infty$ and $\om\in B_\infty$. Let $Z$ be a zero set for $A^p_\om$ such that it is a finite union of separated sequences. Let $0 < \g < 1$ and suppose that there exists another set $Z'$ and a one-to-one correspondence $\s : Z \to Z'$ such that $1-|\s(a)|^2 = \g\left(1-|a|^2\right)$ and $\rho_p(a,\s(a))$ is uniformly bounded away from 1 on $Z$. Then $Z'$ is a zero set for $A^{p/\g}_\om$.
\end{coro}

To deduce Corollary~\ref{coro:zeroPerturb}, note first that since $Z$ is a zero set for $A^p_\om$ by the hypothesis, there exists a harmonic function $h$ such that $\exp(pk_Z-h)$ is integrable with respect to $\om dA$ by Corollary~\ref{theo:zeroset2}. Therefore it suffices to find a harmonic majorant $g$ of $k_{Z'}/\g - k_Z$. Indeed, if such $g$ exists, then $pk_{Z'}/\g-(pg+h) \le pk_Z - h$, and hence $\exp(\frac{p}{\g}k_{Z'}-(pg+h))$ is integrable with respect to $\om dA$, and consequently $Z'$ is a zero set for $A^{p/\g}_\om$ by Corollary~\ref{theo:zeroset2}. A function $g$ with desired properties is constructed in the proof of~\cite[Theorem 4]{L2000}.

Horowitz~\cite{Horzeros,HorFacto} also obtained some results in the spirit of Corollary~\ref{coro:zeroPerturb} describing how the zero sets depend on the parameter $p$. Some of those results were generalized for certain $A^p_\om$ in~\cite{PR2014Memoirs}, and can be further improved by applying Luecking's approach in studying zero sets. In particular, we note that, by applying Proposition~\ref{prop:zeroseq} below instead of~\cite[Lemma~3]{HorFacto}, Theorem~4 and Corollary~2 in~\cite{HorFacto} can be generalized to the case $\om\in B_\infty$ with only minor modifications to the original proofs.

We now turn to consider factorization of functions in $A^p_\om$. By refining Horowitz' original probabilistic argument by Luecking's method to study the zero sets, and then adopting the whole reasoning to the class of weights we are interested in we derive the following factorization result.

\begin{theo}\label{theo:factorization}
Let $0<p<\infty$ and $\om\in B_\infty$ such that the polynomials are dense in $A^p_\om$. Let $f\in A^p_\om$ and $0<p_1,p_2<\infty$ such that $p^{-1}=p_1^{-1}+p_2^{-1}$. Then there exist $f_1\in A^{p_1}_\om$ and $f_2\in A^{p_2}_\om$ such that $f=f_1f_2$ and
    $$
    \|f_1\|_{A^{p_1}_\om}^{p}\|f_2\|_{A^{p_2}_\om}^{p}
    \le\frac{p}{p_1}\|f_1\|_{A^{p_1}_\om}^{p_1}+\frac{p}{p_2}\|f_2\|_{A^{p_2}_\om}^{p_2}
    \le C\|f\|_{A^p_\om}^p\le C\|f_1\|_{A^{p_1}_\om}^{p}\|f_2\|_{A^{p_2}_\om}^{p}
    $$
for some constant $C=C(\om)>0$.
\end{theo}

The density of polynomials is needed as a hypothesis only to guarantee the existence of a dense family of functions with finitely many zeros. Any other requirement implying this property would suffice here. The question of when polynomials are dense in $A^p_\om$ is an old problem and remains unsolved in general, see for example \cite[Section~1.5]{PR2014Memoirs} for basic information and relevant references. However, in certain special cases the density of polynomials can be deduced from an atomic decomposition. For example, \cite[Theorem~4.1]{L1985/2} offers such a decomposition for functions in certain weighted Bergman spaces induced by non-radial weights in terms of the kernel functions of the standard weighted Bergman spaces. Now that these kernels are analytic beyond the boundary of the unit disc, this in turn yields the density of polynomials. Another relevant reference regarding atomic decomposition in the non-radial case is \cite{Constantin}. The argument used in these studies does not work for the class $B_\infty$, and that is understandable because $B_\infty$ contains weights that induce very small Bergman spaces.

We next discuss a specific step in the proofs of the results above and then turn to consider zeros and factorization for $A^p_\om$ induced by a radial weight. The key ingredient in the proofs of Theorems~\ref{theo:zeroset} and~\ref{theo:factorization} is Proposition~\ref{prop:zeroseq} which states that
    $$
    h(z) = \frac{|f(z)|}
        {\prod_{a \in Z} \left\{\left|\frac{a-z}{1-\overline{a}z}\right|
        \exp\left[\frac12\left(1-\left|\frac{a-z}{1-\overline{a}z}\right|^2\right)\right]\right\}},
    $$
induced by $Z\subset \mathcal{Z}(f)$, satisfies $\|f\|_{A^p_\om}\asymp \|h\|_{L^p_\om}$ for each $f \in A^p_\om$. The proof of this fact eventually boils down to showing that
    $$
    R(f)(z)= \int_\D f(w)\frac{\left(1-|z|^2\right)^2}{|1-\overline{z}w|^4}\,dA(w),\quad z\in\D,
    $$
is a bounded operator from $L^q_\om$ into itself for some $q>1$. This is in turn equivalent to the boundedness of the Bergman projection
    $$
    P_\a(f)(z) = \int_\D \frac{f(w)}{(1-z\overline{w})^{2+\a}} \left(1-|w|^2\right)^\a\,dA(w),\quad z\in\D,
    $$
for $\alpha=2$ on certain $L^q$-spaces, and therefore this step is done at once by using the classical characterization of the one-weight inequality for the Bergman projection by Bekoll\'e and Bonami~\cite{B1981,BB1978}. This yields the hypothesis $\om\in B_\infty$ in Theorems~\ref{theo:zeroset} and \ref{theo:factorization}, the proofs of which are presented in Section~\ref{sec:ProofThm1}.

The defect in the hypothesis $\om\in B_\infty$ is that it does not allow $\om$ to vanish in a set of positive measure, neither $\om$ may be small in a relatively large part of each outer annulus of~$\D$. Our next goal is to show that by restricting our consideration to radial weights we can do better and no longer need to require such smoothness. To do this, let $\om$ be a radial weight such that $\widehat{\om}(z)=\int_{|z|}^1\om(s)\,ds>0$ for all $z\in\D$, for otherwise $A^p_\om=\Hcal(\D)$. A radial weight $\om$ belongs to the class~$\DD$ if there exists a constant $C=C(\om)\ge1$ such that the doubling inequality $\widehat{\om}(r)\le
C\widehat{\om}(\frac{1+r}{2})$ is valid for all $0\le r <1$. If there exist $K=K(\om)>1$ and $C=C(\om)>1$ such that $\widehat{\om}(r)\ge C\widehat{\om}\left(1-\frac{1-r}{K}\right)$ for all $0\le r<1$, then $\om\in\Dd$. 
Additionally, we write $\DDD=\DD\cap\Dd$. The classes of weights $\DD$ and $\DDD$ emerge from fundamental questions in operator theory: recently Pel\'aez and R\"atty\"a~\cite{PR2017Proj} showed that the weighted Bergman projection $P_\om$, induced by a radial weight $\om$, is bounded from $L^\infty$ to the Bloch space $\mathcal{B} = \{f \in \mathcal{H}(\D) : \sup_{z\in\D} |f'(z)|(1-|z|) < \infty\}$ if and only if $\om \in \DD$, and further, it is bounded and onto if and only if $\om \in \DDD$.
For further information on these classes of weights, see~\cite{PelSum14,PR2014Memoirs,PR2015Embedding,PR2016TwoWeight}. Moreover, since weights in $\DD$ may very well vanish in a set of positive measure, $\DD$ is not contained in $B_\infty$.

Assume $f\in A^p_\omega$, where $\om\in\DD$. Then
    \begin{equation*}
    \begin{split}
    \|f\|_{A^p_\omega}^p&\ge\int_{\D\setminus
    D\left(0,\frac{1+r}{2}\right)}|f(z)|^p\omega(z)\,dA(z)
    \gtrsim
    M_p^p\left(\frac{1+r}{2},f\right)\widehat{\om}\left(r\right),\quad r\to1^-,
    \end{split}
    \end{equation*}
from which the well known inequality $M_\infty(r,f)\lesssim
M_p(\frac{1+r}{2},f)(1-r)^{-1/p}$ yields
    \begin{equation*}
    M^p_\infty(r,f)\lesssim\frac{\|f\|_{A^p_\omega}^p}{(1-r)\widehat{\om}\left(r\right)},\quad r\to1^-.
    \end{equation*}
Now that $\om\in\DD$, there exist $C=C(\om)>0$ and $\b=\b(\om)>0$ such that $\widehat{\om}\left(r\right)\ge C\widehat{\om}(0)\left(1-r\right)^\b$ by Lemma~\ref{Lemma:weights-in-D-hat}(ii) below. Hence
    $
    |f(z)|\lesssim\|f\|_{A^p_\omega}(1-|z|)^{-\frac{\b+1}{p}}
    $
for all $z\in\D$, and it follows that $\log|f|\in L^1$. Therefore the reasoning to be used in the proof of Proposition~\ref{prop:zeroseq} can be applied in this case also. However, as explained above, the argument relies on the characterization of the one-weight inequality for the Bergman projection by Bekoll\'e and Bonami, which guarantees the boundedness of the auxiliary operator $R:L^q_\om\to L^q_\om$ under the hypothesis $\om\in B_q\subset B_\infty$. But what is actually needed here is to show that
    \begin{equation}\label{eq:step-radial-case}
    \int_\D\left(\int_\D |f(\z)|^\frac{p}{q}\frac{(1-|z|^2)^2}{|1-\overline{z}\zeta|^4}\,dA(\z)\right)^q\om(z)\,dA(z)\lesssim\|f\|_{A^p_\om}^p,\quad f\in A^p_\om,
    \end{equation}
for a sufficiently large $q=q(\om)>p$. The functions involved being analytic, this inequality is better controllable by using Carleson embedding theorems~\cite{PR2015Embedding} rather than weight inequalities for $L^p$-functions~\cite{B1981,BB1978}. By using this approach we will prove the statement of Proposition~\ref{prop:zeroseq} for $\om\in\DD$, and deduce the following result.

\begin{theo}\label{Theorem:D-hat}
The statements in Theorems~\ref{theo:zeroset} and~\ref{theo:factorization} and Corollaries~\ref{theo:zeroset2} and~\ref{coro:zeroPerturb} are valid when the hypothesis $\om\in B_\infty$ is replaced by $\om\in\DD$.
\end{theo}

Details of the deduction yielding this theorem are given in Section~\ref{Sec:radial-zeros-factorization}. The reason why this approach does not yield a better result in the non-radial case is that known Carleson embedding theorems impose growth and/or smoothness restrictions to the weight. In particular, \cite[Theorem~3.1]{Constantin2} states that if $P^+_\eta:L^q_\om\to L^q_\om$ is bounded for some $q>1$ and $\eta>-1$, then $A^p_\om$ is continuously embedded into $L^p_\mu$ if and only if $\mu(\Delta(z,r))\lesssim\om(\Delta(z,r))$ for all $z\in\D$. Each radial weight $(1-|z|)^{-1}\left(\log\frac{e}{1-|z|}\right)^{-\b}$ with $\b>1$ belongs to $B_\infty$, but does not satisfy the Bekoll\'e-Bonami condition, and thus this approach does not give us anything better than what we know already. This and many other possible applications, some of which will appear later in this paper, suggest that it would be desirable to obtain new information on the embedding $A^p_\om\subset L^p_\mu$ when $\om$ is non-radial and $\mu$ is a positive Borel measure on $\D$. In the radial case the hypothesis on the density of polynomials in Theorem~\ref{theo:factorization} is of course always satisfied and can thus be omitted.

The statement in Theorem~\ref{theo:factorization} for $\om\in B_\infty\cup\DD$ significantly improves the main result in \cite[Chapter~3]{PR2014Memoirs} and Horowitz' original result as well because now the factorization is available for the whole class $\DD$ and the constant $C$ appearing in the inequality for the norms is independent of the parameters $p$, $p_1$ and $p_2$.

We mention three immediate consequences of our results so far. The factorization given in Theorem~\ref{Theorem:D-hat} shows that the statement in \cite[Theorem~4.1(iv)]{PR2014Memoirs} concerning the integration operator $T_g(f)(z)=\int_0^z f(\z)g'(\z)\,d\z$ induced by $g\in\mathcal{H}(\D)$ is valid for $\om\in\DD$. The theorem states that $T_g:A^p_\om\to A^q_\om$ is bounded in the upper triangular case $0<q<p<\infty$ if and only if $g\in A^s_\om$ with $\frac1s=\frac1q-\frac1p$. Also the Aleman-Sundberg question, discussed in more detail in \cite[p.~38]{PR2014Memoirs}, on subsets of zero-sets has an affirmative answer when $\om\in\DD$ by Theorem~\ref{theo:zeroset} because in this case each subset of a zero set is always a zero set. Moreover, the statement in \cite[Proposition~7.5]{PelSum14} concerning bounded and compact composition operators acting between different weighted Bergman spaces is valid under the hypotheses $\om\in B_\infty\cup\DD$ and the density of polynomials in $A^p_\om$ by Theorems~\ref{theo:factorization} and~\ref{Theorem:D-hat}. The proposition says that if $0<p,q<\infty$, $n\in\N$ and $\nu$ is any weight, then the composition operator $C_\vp$, defined by $C_\vp(f)=f\circ\vp$, is bounded (resp.compact) from $A^p_\om$ to $A^q_\nu$ if and only if $C_\vp:A^{np}_\om\to A^{nq}_\nu$ is bounded (resp.compact). Therefore one may assume that the both parameters $p$ and $q$ are greater than two and this often simplifies some arguments.

To finish our consideration of zeros and factorization, we give an application of the obtained factorization in the study of the small Hankel operator induced by a conjugate analytic symbol in the upper triangular case. Let $\om$ be a weight such that the reproducing kernels $B^\om_z$ of $A^2_\om$ exist, that is, $f(z)=\langle f,B^\om_z\rangle_{A^2_\om}$ for all $z\in\D$ and $f\in A^2_\om$. For $f\in\Hcal(\D)$ consider the small Hankel operator
    $$
    h^\om_{\overline{f}}(g)(z)=\overline{P_\om}(\overline{f}g)=\int_\D\overline{f(\z)}g(\z)B^\om_z(\z)\om(\z)\,dA(\z),\quad z\in\D.
    $$

\begin{coro}\label{corollary:Hankel}
Let $1<q<p<\infty$ and $1<s<\infty$ such that $\frac1s=\frac1q-\frac1p$, and let $\om\in\DD$. Then $h^\om_{\overline{f}}:A^p_\om\to\overline{A^q_\om}$ is bounded if and only if $f\in A^s_\om$.
\end{coro}

This corollary combined with our earlier observation on $T_g$ acting from $A^p_\om$ to $A^q_\om$ confirms the well known phenomenon that in many cases the boundedness of the integration operator and the small Hankel operator are characterized by the same condition.
Corollary~\ref{corollary:Hankel} extends the results of Pau and Zhao~\cite{PauZhao} to the case of Bergman spaces of one complex variable induced by doubling weights.

Hankel, integration and composition operators are among the most studied objects in operator theory of analytic function spaces and the literature concerning the subject is vast. Since none of these operators is in the main focus of the present paper, we invite the reader to see \cite{CowenMac95,Shapiro93} for the theory of composition operators, and \cite{PR2016Toeplitz} for the case $\DD$, \cite{Aleman,AlemanCima} for integration, and~\cite{Peller2003,Zhu} 
for Hankel. For some recent developments on Hankel operators in Bergman spaces, see~\cite{PauZhaoZhu} and references therein.

Our next goal is to study dominating sets for $A^p_\om$ in order to obtain a sufficient condition for a positive Borel measure $\mu$ on $\D$ to be a sampling measure for $A^p_\om$. A measurable set $G=G(f)\subset\D$ is called a dominating set for $f\in A^p_\om$ if there exists $\d=\d(G)>0$ such that
    $$
    \int_G|f(z)|^p\om(z)\,dA(z)\ge\d\|f\|_{A^p_\om}^p.
    $$
If this inequality is valid for all $f\in A^p_\om$, then $G$ is called a dominating set for $A^p_\om$. To state the results, some more notation is needed. The reproducing kernels of the Hilbert space $A^2_\om$ induced by a radial weight $\om$ are given by
    $$
    B^\om_z(\z)=\sum_{n=0}^\infty\frac{(\overline{z}\z)^n}{2\om_{2n+1}},\quad z,\z\in\D,
    $$
where $\om_x=\int_0^1s^{x}\om(s)\,ds$ for all $-1<x<\infty$, and each $f\in A^1_\om$ satisfies
    \begin{equation}\label{eq:reproducing-formula}
    f(z)=\langle f,B^\om_z\rangle_{A^2_\om}=\int_\D f(\z)\overline{B_z^\om(\z)}\om(\z)\,dA(\z),\quad z\in\D.
    \end{equation}
Write
	$$
	K_z^\om(\z)=\frac{|B_z^\om(\z)|^2}{\|B_z^\om\|^2_{A^2_\om}},\quad z,\z\in\D.
	$$
Our first result on dominating sets reads as follows.

\begin{theo}\label{theo:dominating-set-for-f}
Let $0<q<p<\infty$ and $\om\in\DD$. Then there exists a constant $C=C(\om,q)>0$ such that
    \begin{equation}\label{eq:PointEval}
    |f(z)|^q \le C \int_\D |f(\z)|^q K_z^\om(\z) \om(\z)\,dA(\z),\quad f\in A^q_\om,\quad z \in \D.
    \end{equation}
Moreover, if $f\in A^p_\om$ and $E = E(\e,q,f)$ is the set of points $z\in\D$ for which
    \begin{equation}\label{eq:BadPoints}
    |f(z)|^q \leq \e \int_\D |f(\z)|^q K^\om_z(\z)\om(\z)\,dA(\z),
    \end{equation}
then there exists a constant $C = C(p/q,\om)$ such that
    \begin{equation}\label{eq:BadPointsEstimate}
    \int_E |f(z)|^p\om(z)\,dA(z) \le (C\e)^{p/q} \|f\|^p_{A^p_\om}.
    \end{equation}
Therefore $\e > 0$ may be chosen such that $G=\D\setminus E$ satisfies
    \begin{equation*}
    \int_G |f(z)|^p\om(z)\,dA(z) \ge \frac12 \|f\|^p_{A^p_\om},\quad f\in A^p_\om,
    \end{equation*}
and thus $G$ is a dominating set for $f\in A^p_\om$ if $\e>0$ is sufficiently small.
\end{theo}

The inequality \eqref{eq:PointEval} with $C=1$ is easy to establish if $q\ge1$. Namely, an application of the reproducing formula \eqref{eq:reproducing-formula} to $fB^\om_z$ gives
	$$
	f(z)=\int_\D f(\z)K^\om_z(\z)\om(\z)\,dA(\z),\quad z\in\D.
	$$
This gives the assertion for $q=1$, which together with H\"older's inequality can be used to establish the claim in the case $q>1$. This kind of reasoning does not seem to work for $0<q<1$, and we will argue differently; we first use the subharmonicity of $|f|^q$ together with the fact that $|B^\om_z(\z)|\asymp B^\om_z(z)$ for all $\z\in\Delta(z,r)$ and $z\in\D$ if $r=r(\om)\in(0,1)$ is sufficiently small~\cite[Lemma~8]{PRS2015Berezin}, and then employ a proof of a Carleson embedding theorem for specific subharmonic functions~\cite[Theorem~3.3]{PelSum14}.

The estimate \eqref{eq:BadPointsEstimate} can be obtained quite easily for $\om\in\mathcal{D}\subset\DD$ by using the boundedness of the maximal Bergman projection $P^+_\om$ on $L^p_\om$ for each $p>1$, but the general case $\om\in\DD$ is more laborious and relies on the $L^p$-estimates of the kernel functions~$B_z^\om$ given in \cite{PR2016TwoWeight}. The special case of Theorem~\ref{theo:dominating-set-for-f} concerning standard weighted Bergman spaces can be found in \cite[Lemma~2]{L2000}. The proof there is different and does not carry over to the situation of Theorem~\ref{theo:dominating-set-for-f}.

The proof of \eqref{eq:BadPointsEstimate} in Theorem~\ref{theo:dominating-set-for-f} does not work for $p=q$. In that case, it is natural to replace the the right hand side of~\eqref{eq:BadPoints} by an average over a subset of a pseudohyperbolic disc centered at $z$. To state the result some notation is needed. Let $r\in(0,1)$, $\nu$ a positive Borel measure on $\D$ and $E(z)\subset\Delta(z,r)$ such that $\nu(E(z))>0$ for all $z\in\D$. Define
    $$
    Q(f)(z)=\frac1{\nu(E(z))} \int_{E(z)} |f(\z)|^p\nu(\z)\,dA(\z), \quad z\in\D,
    $$
and $E=E(f,\nu)=\{z\in\D:|f(z)|^p\le\e Q(f)(z)\}$.

\begin{theo}\label{theo:dominating-set-for-f-2}
Let $0<p<\infty$, $0<r<1$ and $\om\in\DD$. Then there exists a constant $C=C(r,\om)>0$ such that
    $$
    \int_E|f(z)|^p\om(z)\,dA(z)\le C\e\|f\|^p_{A^p_\om},
    $$
and thus $\D\setminus E$ is a dominating set for $f\in A^p_\om$ if $\e=\e(r,\om)>0$ is sufficiently small.
\end{theo}

Special cases of Theorem~\ref{theo:dominating-set-for-f-2} with $\nu$ being the Lebesgue measure can be found in~\cite[Lemmas~2 and~3]{L1981}. Our proof is different from these results and relies on Carleson measures.

Since our main results so far concern also certain non-radial weights, it is reasonable to discuss that aspect of Theorems~\ref{theo:dominating-set-for-f} and~\ref{theo:dominating-set-for-f-2} as well. The first obstruction in the proof of Theorem~\ref{theo:dominating-set-for-f} for non-radial weights is the pointwise estimate $|B^\om_z(\z)|\asymp B^\om_z(z)$ which does not have a known sufficiently general non-radial extension. The second problem arises with Carleson measures, and finally the lack of satisfactory norm estimates for kernel functions prevents our reasoning from carrying over to the non-radial case all together. The situation of Theorem~\ref{theo:dominating-set-for-f-2} is better because of \cite[Lemmas~1 and~2]{L1984} and the proof of \cite[Theorem~3.9]{L1985/2}, though the weight $\om$ in \cite{L1984} is rather particular (but the domain lies in $\C^n$ and is quite general). A careful inspection of the proof of \cite[Theorem~9]{PR2015Embedding} shows that the argument used to obtain Theorem~\ref{theo:dominating-set-for-f-2} works for weights $\om$ that are doubling in Carleson squares, denoted by $\om\in\DD(\D)$ and defined in detail below, if $N:A^p_\om\to L^p_\om$ is bounded. This raises the question if the maximal operator $N(f)(z)=\sup_{\z\in\G(z)}|f(\z)|$, where
    \begin{equation*}\label{eq:gammadeuintro}
    \Gamma(\z)=\left\{z\in \D:\,|\t-\arg
    z|<\frac12\left(1-\frac{|z|}{r}\right)\right\},\quad
    \z=re^{i\theta}\in\D\setminus\{0\},
    \end{equation*}
are non-tangential approach regions with vertexes inside the disc \cite[Chapter~4.1]{PR2014Memoirs}, is bounded from $A^p_\om$ to $L^p_\om$ when $\om\in\DD(\D)$. Unfortunately, we do not know the answer to this question.

We next proceed to study dominating sets for the whole space $A^p_\om$. For $1<q<\infty$ and a (almost everywhere) positive weight $\om$, write $\om\in C_q$ if for some some (equivalently for each) $r\in(0,1)$, there exists a constant $C=C(q,r,\om)>0$ such that
    \begin{equation*}
    \left(\int_{\Delta(z,r)}\om(z)\,dA(z)\right)^\frac1q\left(\int_{\Delta(z,r)}\om(z)^{-\frac{q'}{q}}\,dA(z)\right)^\frac1{q'}\le C |\Delta(z,r)|,\quad z\in\D,
    \end{equation*}
and set $C_\infty=\cup_{q>1}C_q$. Luecking \cite[Theorem~3.9]{L1985/2} showed that if $G\subset\D$ is measurable, $0<p<\infty$ and $\om\in C_\infty$ such that
    \begin{equation}\label{Eq:Luecking-dominating}
    |G\cap \Delta(z,r)|\ge\delta|\Delta(z,r)|,\quad z\in\D,
    \end{equation}
for some $\d>0$ and $r\in(0,1)$, then $G$ is a dominating set for $A^p_{\om}$. This condition is equivalent to the existence of a constant $\delta_0=\delta_0>0$ such that $|G\cap S|>\delta_0|S|$ for all Carleson squares~$S$.
One can also replace the pseudohyperbolic disc $\Delta(a,r)$ by a suitable Euclidean disc, for example, $D(a, \eta(1-|a|))$ for a fixed $0<\eta<1$ would work here. For the proofs of these equivalences, see~\cite{L1981}. The condition \eqref{Eq:Luecking-dominating} is known to be also a necessary condition for $G$ to be a dominating set for $A^p_{\om}$ if $\om$ is one of the standard weights by \cite{L1981}.

The existing literature does not offer results concerning the converse statement of the above-mentioned result in the non-radial case. We next turn our attention to this matter, and to do it we write $\om\in\DD(\D)$ if there exists $C=C(\om)>0$ such that $\om(S(a))\le C\om(S(\frac{1+|a|}{2}e^{i\arg a}))$ for all $a\in\D\setminus\{0\}$. It is easy to see that each $\om\in\DD(\D)$ satisfies $\om(S(a'))\le C(C+1)\om(S(a))$ for all $a,a'\in\D\setminus\{0\}$ with $|a'|=|a|$ and $\arg a'=\arg a\pm(1-|a|)$. Therefore $\om(S(a))\lesssim\om(S(b))$ whenever $|b|=\frac{1+|a|}{2}$ and $S(b)\subset S(a)$. Moreover, it is obvious that radial weights in $\DD(\D)$ form the class $\DD$.

\begin{theo}\label{theo:dominating-sets-non-radial}
Let $0<p<\infty$ and $\om\in\DD(\D)$. If $G$ is a dominating set for $A^p_\om$, then there exists a constant $\delta > 0$ such that
\begin{equation}\label{eq:delta-condition-set-G-non-radial}
    \om(G\cap S)>\delta\om(S)
\end{equation}
for all Carleson squares $S$.
\end{theo}

The proof of Theorem~\ref{theo:dominating-sets-non-radial} is based on characterizations of weights in $\DD(\D)$ given in Lemma~\ref{Lemma:test-functions-non-radial} below and appropriately chosen test functions. By combining Theorem~\ref{theo:dominating-sets-non-radial} with \cite[Theorem~3.9]{L1985/2} and imposing severe additional hypothesis on $\om$ one can certainly obtain a characterization of dominating sets in the non-radial case but because of these extra assumptions the resulting description is far from being satisfactory. The approach involving the Lebesgue measure and yielding \eqref{Eq:Luecking-dominating} is natural and has been efficiently used in \cite{L1981}, \cite{L1984} and \cite{L1985/2}, but it seems that the arguments used there are not adoptable as such to prove \eqref{eq:delta-condition-set-G-non-radial} to be a sufficient condition. It is of course equally natural to measure the set $G$ as in \eqref{eq:delta-condition-set-G-non-radial} by using the weight $\om$ itself that induces the space. Moreover, the studies on Carleson measures~\cite{PelSum14}, \cite{PR2014Memoirs}, \cite{PR2015Embedding}, \cite{PelRatSie2015} strongly support the use of Carleson squares instead of pseudohyperbolic discs as testing sets, at least when $\om$ induces a very small weighted Bergman space. It is also worth noticing that the hypothesis $\om\in\DD(\D)$ allows $\om$ to vanish in a relatively large part of each outer annulus of $\D$, meanwhile weights in $C_\infty$ may not have this property because of the negative power $-\frac{q'}{q}$ appearing in the definition of $C_q$. Therefore a complete solution to the question of when a set $G$ is a dominating set for $A^p_\om$ remains as an open problem in both non-radial weight classes $C_\infty$ and $\DD(\D)$.

Finally, we turn our attention to sampling. A positive Borel measure on $\D$ is a sampling measure for $A_\om^p$ if
    \begin{equation*}
    \int_\D |f(z)|^p \,d\mu(z) \asymp \|f\|_{A_\om^p}^p, \quad f \in A_\om^p.
    \end{equation*}
The measures $\mu$ satisfying the inequality ''$\lesssim$'' are the $p$-Carleson measures for $A_\om^p$. These measures in the case $\om\in\DD$ have been studied in \cite{PR2014Memoirs,PR2015Embedding,PelRatSie2015}, and can be characterized in terms of the weighted maximal function
    $$
    M_\om(\mu)(z) = \sup_{S\ni z}\frac{\mu(S)}{\om(S)}, \quad z \in \D:
    $$
$\mu$ is a $p$-Carleson measure for $A_\om^p$ if and only if $M_\om(\mu) \in L^\infty$, and $\|Id\|_{A^p_\om\to L^p_\mu}^p\asymp\|M_\om(\mu)\|_{L^\infty}$. Therefore these measures are independent of $p$ for each $\om\in\DD$. Before stating our result, some more notation is needed. For a positive Borel measure $\mu$ on $\D$ and $r\in(0,1)$, let $k_r(z) = \mu(\Delta(z,r))/\om(\Delta(z,r))$ for all $z\in\D$.
The next theorem describes a condition sufficient to guarantee that a positive Borel measure $\mu$ is a sampling measure for $A^p_\om$. Recall that the hypothesis $\mu(\Delta(a,r))\lesssim\om(\Delta(a,r))$ characterizes Carleson measures for certain $A^p_\om$ as mentioned in the paragraph just after Theorem~\ref{Theorem:D-hat}.

\begin{theo}\label{theo:sampling-suff}
Let $0<p<\infty$, $\e>0$ and either $\om\in\DD$ such that $\om>0$ almost everywhere on $\D$, and $\mu$ a $p$-Carleson measure for $A^p_\om$, or $\om\in C_\infty$ and $\mu(\Delta(a,r))\lesssim\om(\Delta(a,r))$ for all $a\in\D\setminus\{0\}$. Then there exists an $r \in (0,1)$
such that $\mu$ is a sampling measure for $A^p_\om$ whenever the set $G = \{z \in \D : k_r(z) > \e\|M_\om(\mu)\|_{L^\infty}\}$ is a dominating set for $A^p_\om$.
\end{theo}

The proof of Theorems~\ref{theo:sampling-suff} follows the ideas of Luecking~\cite{L1985}, but a crucial step in the case of $\DD$ relies on the characterization of Carleson measures. Additionally, the presence of the weight $\om$ also makes the use of convenient changes of variables and automorphisms difficult and thus forces us to make some more delicate observations.

One can readily see from the proof that in the case of $\DD$ one may omit the extra hypothesis on the positivity of $\om$ by replacing $k_r$ by $k^\star_r(z)=\mu(\Delta(z,r))/\om(S(z))$. In this case the set $G$ may become essentially smaller and thus it being dominating set would be a stronger hypothesis.

Let $(\mu_n)$ be a sequence of measures on $\D$. We say that $(\mu_n)$ converges weakly to a measure $\mu$, denoted by $\mu_n \rightharpoonup \mu$, if
    $$
    \int_\D h(z)\,d\mu_n(z) \to \int_\D h(z)\,d\mu(z)
    $$
for all $h$ in the class $C_c(\D)$ of nonnegative continuous compactly supported functions in $\D$. The following result is a generalization of~\cite[Theorem~1]{L2000}, and completes our study of sampling measures.

\begin{theo}\label{theo:w-conv-meas}
Let $0<p<\infty$ and $\om\in\DD$, and let $(\mu_n)$ be a sequence of $p$-Carleson measures for $A^p_\om$ such that $\sup_n\|M_\om(\mu_n)\|_{L^\infty}< \infty$. Then $(\mu_n)$ has a weakly convergent subsequence.

Further, if $\mu_n \rightharpoonup \mu$, then
    \begin{equation}\label{eq:w-conv-meas}
    \lim_{n \to \infty} \int_\D |f(z)|^p \,d\mu_n(z)
    = \int_\D |f(z)|^p \,d\mu(z),\quad f \in A_\om^p,
    \end{equation}
and $\mu$ is a $p$-Carleson measure for $A^p_\om$ with $\|Id\|_{A^p_\om \to L^p_\mu}\le\liminf_{n\to\infty}\|Id\|_{A^p_\om\to L^p_{\mu_n}}$. Furthermore, if $\mu_n$ are sampling measures with sampling constants at most $\Lambda > 0$, then $\mu$ is also a sampling measure with a sampling constant at most~$\Lambda$.
\end{theo}

The proof of Theorem~\ref{theo:w-conv-meas} follows the lines of that of~\cite[Theorem~1]{L2000}, but the crucial step which differs from the original argument relies on Carleson embedding for tent spaces given in \cite[Theorem~9]{PR2015Embedding}. These tent spaces of measurable functions are defined by using the maximal function $N(f)$.

Luecking~\cite{L2000} characterized the sampling measures for Bergman spaces induced by standard weights. However, the methods used there do not generalize for weights in the class $\DD$ because the weights in $\DD$ can be such that compositions of functions in $A^p_\om$ with M\"obius transformations cannot be controlled in norm. Thus, the problem of characterizing sampling measures for $A^p_\om$ when $\om \in \DD$ remains open. However, the closely related sampling sequences were characterized by Seip~\cite{Seip} in small weighted Bergman spaces induced by weights admitting a pointwise doubling condition in $\om$ instead of $\widehat{\om}$.

\section{Zeros and factorization when $\om\in B_\infty$}\label{sec:ProofThm1}

We begin with briefly analyzing the classes $B_q$ and $B_\infty$. For $-1<\alpha<\infty$, write $dA_\alpha(z)=(1-|z|^2)^\alpha\,dA(z)$ for short. Bekoll\'e and Bonami~\cite{B1981,BB1978} showed that for $1<p<\infty$, $P_\alpha:L^p_{\om_{[\alpha]}}\to L^p_{\om_{[\alpha]}}$ is bounded if and only if
    \begin{equation*}
    BB_{p,\alpha}(\om)=\sup_{S}\left(\frac{1}{A_\alpha(S)}\int_{S}\om(z)\,dA_\a(z)\right)
        \left(\frac{1}{A_\alpha(S)}\int_{S}\om(z)^{-\frac1{p-1}}\,dA_{\a}(z)\right)^{p-1}
        <\infty,
    \end{equation*}
where $A_\alpha(E)=\int_EdA_\alpha$ for each measurable set $E\subset\D$. Denote by $BB_{p,\a}$ the set of these weights, and write $BB_{\infty,\a}=\cup_{1<p<\infty}BB_{p,\a}$.
By comparing this to the definition of $B_q$, we see that $\om\in B_q$ if and only if $\om_{[2q-2]}\in BB_{q,2}$.
Moreover, it is known that $\om\in BB_{\infty,\a}$ if and only if there exist $\delta\in(0,1)$ and $C>0$ such that
    $$
    \frac{A_\alpha(E)}{A_\alpha(S)}\le C\left(\frac{\om(E)}{\om(S)}\right)^\d,\quad E\subset S,
    $$
for all Carleson squares $S$. This condition corresponds to the characterization of the restricted weak-type inequality for the Hardy-Littlewood maximal operator by Kerman and Torchinsky~\cite{KT1980}. For more on $A_\infty$-conditions, see \cite{DMO2016} and the references therein.

\begin{prop} Let $\om$ be an almost everywhere strictly positive weight. Then the following assertions hold:
\begin{itemize}
\item[\rm(i)] $B_p\subsetneq B_q$ for $1<p<q<\infty$;
\item[\rm(ii)] $\om\in B_q$ if and only if $W_{q,\om}\in B_{q'}$, where $W_{q,\om}(z)=(\om(z)^\frac1q(1-|z|)^2)^{-q'}$;
\item[\rm(iii)] If $\om\in B_q$, then
    $$
    \frac{A_2(E)}{|S|^2}\le B_q(\om)^\frac1q\left(\frac{\om_{[2q]}(E)}{\om_{[2q]}(S)}\right)^\frac1q,\quad E\subset S\subset\D;
    $$
\item[\rm(iv)] If there exist $q>1$, $\d\in(\frac1q,1)$ and $C>0$ such that
    \begin{equation}\label{3}
    \frac{A_2(E)}{|S|^2}\le C\left(\frac{\om_{[2q]}(E)}{\om_{[2q]}(S)}\right)^\d,\quad E\subset S\subset\D,
    \end{equation}
then $\om\in B_q$;
\item[\rm(v)] $B_q(\om)\lesssim BB_{q,0}(\om_{[\alpha]})$ for all $\alpha\le2q$.
\end{itemize}
\end{prop}

\begin{proof}
H\"older's inequality and the inequality $(1-|z|)^2\lesssim|S|$ for $z\in S$ imply $B_q(\om)\lesssim B_p(\om)$ for $1<p<q<\infty$, and thus $B_p\subset B_q$. The inclusion is seen to be strict by considering standard power weight $(1-|z|)^\alpha$ with $p<\a+1<q$. Thus (i) is satisfied. Moreover, since $W_{q,\om}(z)(1-|z|^2)^{2q'}=\om(z)^{-\frac1{q-1}}$ and $W_{q,\om}(z)^{-\frac1{q'-1}}=\om(z)(1-|z|^2)^{2q}$, the assertion (ii) follows by the definition of $B_q$.

To prove (iii), assume $\om\in B_q$. Then, for each $E\subset S$, H\"older's inequality and the definition of $B_q$ give
    \begin{equation*}
    \begin{split}
    A_2(E)
    &\le\left(\int_E\om_{[2q]}(z)\,dA(z)\right)^\frac1q\left(\int_S\om(z)^{-\frac{q'}{q}}\,dA(z)\right)^\frac1{q'}\\
    &\le B_q(\om)^\frac1q|S|^2\left(\frac{\int_E\om_{[2q]}(z)\,dA(z)}{\int_S\om_{[2q]}(z)\,dA(z)}\right)^\frac1q,
    \end{split}
    \end{equation*}
and thus
    $$
    \frac{A_2(E)}{|S|^2}\le B_q(\om)^\frac1q\left(\frac{\int_E\om_{[2q]}(z)\,dA(z)}{\int_S\om_{[2q]}(z)\,dA(z)}\right)^\frac1q
    $$
as claimed.

To see (iv), assume \eqref{3} and write $E_\lambda=\{z\in S:\om_{[2q-2]}(z)<1/\lambda\}$ for all $\lambda>0$. Then
    \begin{equation*}
    \lambda\om_{[2q]}(E_\lambda)\le A_2(E_\lambda)\le C|S|^2\left(\frac{\om_{[2q]}(E_\lambda)}{\om_{[2q]}(S)}\right)^\d,
    \end{equation*}
and hence
    $$
    \om_{[2q]}(E_\lambda)\le C^\frac1{1-\delta}\lambda^{-\frac{1}{1-\d}}\left(\frac{|S|^2}{\om_{[2q]}(S)^\d}\right)^\frac1{1-\d}.
    $$
Therefore, by denoting $M=|S|^2/\om_{[2q]}(S)$, we deduce
    \begin{equation*}
    \begin{split}
    \int_S\om(z)^{-\frac1{q-1}}\,dA(z)
    &=\int_S\frac{\om_{[2q]}(z)}{\left(\om_{[2q-2]}(z)\right)^{q'}}\,dA(z)
    =q'\int_0^\infty\lambda^{q'-1}\om_{[2q]}(E_\lambda)\,d\lambda\\
    &=q'\left(\int_0^M+\int_M^\infty\right)\lambda^{q'-1}\om_{[2q]}(E_\lambda)\,d\lambda\\
    &\le\om_{[2q]}(S)M^{q'}+C^\frac1{1-\d}\left(\frac{|S|^2}{\om_{[2q]}(S)^\d}\right)^\frac1{1-\d}q'\int_M^\infty\lambda^{q'-1-\frac1{1-\d}}\,d\lambda\\
    &=\frac{|S|^{2q'}}{\left(\om_{[2q]}(S)\right)^{q'-1}}+\frac{C^\frac1{1-\d}q'}{\frac1{1-\d}-q'}\left(\frac{|S|^2}{\om_{[2q]}(S)^\d}\right)^\frac1{1-\d}M^{q'-\frac1{1-\d}}\\
    &=\left(1+\frac{q(1-\d)C^{\frac1{1-\d}}}{q\d-1}\right)\frac{|S|^{2q'}}{\left(\om_{[2q]}(S)\right)^{q'-1}}
    \end{split}
    \end{equation*}
and it follows that $\om\in B_q$.

To see (v), let $q\ge\frac{\a}{2}\ge0$. Then
    \begin{equation*}
    \begin{split}
    &\om_{[2q]}(S)
    \left(\int_{S} \om(z)^{-\frac1{q-1}}\,dA(z)\right)^{q-1}\\
    &=\left(\int_{S}\om_{[\a]}(z)(1-|z|^2)^{2q-\a}\,dA(z)\right)
    \left(\int_{S} \om_{[\a]}(z)^{-\frac1{q-1}}(1-|z|^2)^\frac{\a}{q-1}\,dA(z)\right)^{q-1}\\
    &\le\left(\int_{S}\om_{[\a]}(z)\,dA(z)\right)
    \left(\int_{S} \om_{[\a]}(z)^{-\frac1{q-1}}\,dA(z)\right)^{q-1}(1-|a|^2)^{2q},
    \end{split}
    \end{equation*}
and the assertion follows.
\end{proof}

The following key proposition is a direct generalization of~\cite[Theorem~2]{L1996}.

\begin{prop}\label{prop:zeroseq}
Let $0<p<\infty$ and $\om \in B_\infty$. Let $f \in \Hcal(\D)$ and $Z \subset \Zcal(f)$. Then the function
    $$
    h(z) = \frac{|f(z)|}
        {\prod_{a \in Z} \left\{\left|\frac{a-z}{1-\overline{a}z}\right|
        \exp\left[\frac12\left(1-\left|\frac{a-z}{1-\overline{a}z}\right|^2\right)\right]\right\}},\quad z\in\D,
    $$
belongs to $L^p_\om$ if and only if $f \in A^p_\om$. Moreover, there exists a constant $C=C(\om)>0$ such that
    $$
    \|f\|_{A^p_\om}^p\le\|h\|_{L^p_\om}^p\le C\|f\|_{A^p_\om}^p, \quad f \in A^p_\om.
    $$
\end{prop}

\begin{pro}
Since $1-x^2 < 2\log\frac1x$ for $x \in (0,1)$, each factor in the denominator of $h$ is less than $1$. Thus the ''only if'' part along with the first inequality is trivial, and for the converse, it suffices to consider the case where $Z = \Zcal(f)$.

To see the second inequality, we start by constructing the denominator of $h$. For $f \in \Hcal(\D)$ with $f(0) \neq 0$ and the zero sequence $\Zcal(f)$, Jensen's formula gives
    \begin{equation}\label{eq:Jensen}
    \log|f(0)| + \sum_{a \in \Zcal(f)}\log\frac{r}{|a|}\chi_{[|a|,1)}(r)
    = \frac1{2\pi} \int_0^{2\pi} \log\left|f\left(re^{i\theta}\right)\right|\,d\theta.
    \end{equation}
If $f\in A^p_\om$ with $\om\in B_\infty$, then $\log|f|$ is area integrable. Indeed, since $\om \in B_\infty$, there exists $q=q(\om)> 1$ such that $\om \in B_q$, and hence
\begin{equation*}
    \begin{split}
    \int_\D \log|f(z)|\,dA(z)
    &= \int_\D \log\left(|f(z)|\om(z)^{1/p}\right)\,dA(z) - \int_\D \log\om(z)^{1/p}\,dA(z) \\
    &\leq \frac1p\|f\|_{A^p_\om}^p + \frac1p \int_\D \log\frac1{\om(z)}\,dA(z) \\
    &\leq \frac1p\|f\|_{A^p_\om}^p + \frac{q-1}p \int_\D \om(z)^{-1/(q-1)}\,dA(z)
    \end{split}
\end{equation*}
because $\log x \leq \frac1\d x^\d$ for all $x,\d > 0$. But the last integral is convergent because $\om \in B_q$, and hence $\log|f| \in L^1$. Therefore, as in~\cite[p.~348]{L1996}, an integration of \eqref{eq:Jensen} with respect to $2r\,dr$ now gives
    $$
    \log|f(0)| + \sum_{a \in \Zcal(f)}\left(\log\frac{1}{|a|} - \frac12\left(1-|a|^2\right)\right)
    = \int_\D \log|f(w)|\,dA(w),
    $$
and applying this to $w \mapsto f\left(\vp_z(w)\right)$ yields
\begin{equation*}
    \log|f(z)| + \sum_{a \in \Zcal(f)}\left[\log\frac{1}{|\vp_a(z)|} - \frac12\left(1-\left|\vp_a(z)\right|^2\right)\right]
    = \int_\D \log|f(w)| \left|\vp_z'(w)\right|^2\,dA(w)
\end{equation*}
for $z \notin \Zcal(f)$. Exponentiating this and applying Jensen's inequality then gives
\begin{equation}\label{eq:gCalc}
    \frac{|f(z)|}
        {\prod_{a \in \Zcal(f)} \left\{\left|\vp_a(z)\right|
        \exp\left[\frac12\left(1-\left|\vp_a(z)\right|^2\right)\right]\right\}}
        \leq \left(\int_\D |f(w)|^\d \left|\vp_z'(w)\right|^2\,dA(w)\right)^{1/\d}
\end{equation}
for any $\d>0$.

We next consider the linear integral operator
    $$
    R(f)(z)=\int_\D f(w)\frac{\left(1-|z|^2\right)^2}{|1-\overline{z}w|^4}\,dA(w),\quad z\in\D,
    $$
appearing on the right-hand side of~\eqref{eq:gCalc}, and will show its boundedness on $L^q_\om$. To do this, write $f_{-2}(z) = f(z)\left(1-|z|^2\right)^{-2}$ and $\om_{[x]}(z) = \om(z)\left(1-|z|^2\right)^{x}$ for short. Then $\|f\|_{L^q_\om} = \left\|f_{-2}\right\|_{L^q_{\om_{[2q]}}}$
and $\|R(f)\|_{L^q_\om} = \left\|P^+_2\left(f_{-2}\right)\right\|_{L^q_{\om_{[2q]}}}$, where
    $$
    P^+_\a(f)(z) = \int_\D \frac{f(w)}{|1-z\overline{w}|^{2+\a}} \left(1-|w|^2\right)^\a\,dA(w),\quad z\in\D,
    $$
is the maximal Bergman projection. Hence, $R : L^q_\om \to L^q_\om$ is bounded if and only if $P^+_2 : L^q_{\om_{[2q]}} \to L^q_{\om_{[2q]}}$ is bounded, and the corresponding operator norms are equal. By the characterization of Bekoll\'e and Bonami~\cite{B1981,BB1978}, this is equivalent to
    $$
    \int_{S(a)} \om_{[2q-2]}(z)(1-|z|)^2\,dA(z)
        \left(\int_{S(a)} \om_{[2q-2]}(z)^{-\frac1{q-1}}(1-|z|)^2\,dA(z)\right)^{q-1}
        \lesssim (1-|a|)^{4q},\quad a\in\D\setminus\{0\},
    $$
that is,
    $$
    \int_{S(a)} \om(z)(1-|z|)^{2q}\,dA(z)
        \left(\int_{S(a)} \om(z)^{-\frac1{q-1}}\,dA(z)\right)^{q-1}
        \lesssim (1-|a|)^{4q},\quad a\in\D\setminus\{0\}.
    $$
Thus, $R : L^q_\om \to L^q_\om$ is bounded if and only if $\om \in B_q$. Moreover, $\|R\|_{L^q_\om \to L^q_\om}\le CB_q(\om)^{\max\{1,\frac1{q-1}\}}$, where $C=C(\om,q)$, by \cite[Theorem~1.5]{PottRegueraJFA13}.

Let now $Z = \Zcal(f)$. Clearly, $f \in A^p_\om$ implies $|f|^\d \in L^{p/\d}_\om$. Choose $\d = p/q$, so that $p/\d = q$. By the boundedness of $R$ and~\eqref{eq:gCalc}, we obtain
    $$
    \|h\|_{L^p_\om}^p\le\left\|R\left(|f|^\frac{p}{q}\right)^{\frac{q}{p}}\right\|_{L^p_\om}^p = \left\|R\left(|f|^\frac{p}{q}\right)\right\|^q_{L^{q}_\om}\le C^qB_q(\om)^{\max\{q,q'\}}\|f\|^p_{A^p_\om},
    $$
where $C=C(q,\om)$. Since $q=q(\om)$, the norm estimate we are after follows.
\end{pro}

\begin{proofof}{Theorem~\ref{theo:zeroset}}
We start the proof by using the function $h$ in Proposition~\ref{prop:zeroseq} to find the analytic function $F$ in cases (c) and (d) of the theorem. To do this, we first make some observations about the functions $\psi$, $k$ and $W$ defined in Section~\ref{sec:Intro}.

Standard estimates using the power series of the exponential show that
    $$
    \left|\overline{a}\frac{a-z}{1-\overline{a}z} \exp\left(1-\overline{a}\frac{a-z}{1-\overline{a}z}\right) - 1\right| = O\left((1-|a|)^2\right), \quad |a| \to 1^-,
    $$
uniformly on compact subsets of $\D$, and thus $\psi\in\Hcal(\D)$ if $\sum_{a \in Z} (1-|a|)^2 < \infty$. This sum converges for $Z \subset \Zcal(f)$ whenever $\log |f|$ is integrable on $\D$. This, in turn, is true for any $f \in A^p_\om$ with $\om\in B_\infty$ by the proof of Proposition~\ref{prop:zeroseq}.

A direct calculation shows that
\begin{align*}
    \re\left(1-\overline{a}\frac{a-z}{1-\overline{a}z}\right)
        &= \frac12\left(1-|a|^2\right) + \frac12\left(1-\left|\frac{a-z}{1-\overline{a}z}\right|^2\right) + \frac{|z|^2}2\frac{\left(1-|a|^2\right)^2}{|1-\overline{a}z|^2},
\end{align*}
and therefore
\begin{align*}
    |\psi_Z(z)| &= \prod_{a \in Z} |a|e^{\left(1-|a|^2\right)/2}
        \prod_{a \in Z}\left\{\left|\vp_a(z)\right| \exp\left[\frac12\left(1-\left|\vp_a(z)\right|^2\right)\right]\right\} \exp\left[\frac{|z|^2}2 \sum_{a \in Z} \frac{\left(1-|a|^2\right)^2}{|1-\overline{a}z|^2}\right],
\end{align*}
see~\cite{L1996} for details. The last exponential is now exactly the function $W_Z$ defined in the first section, and the second product is the denominator of $h$ in Proposition~\ref{prop:zeroseq}. The function $h$ can thus be written as
    $$
    h(z) = C\left|\frac{f(z)}{\psi_Z(z)}\right|W_Z(z),
    $$
where $C = \prod_{a \in Z} |a|e^{\left(1-|a|^2\right)/2}$. We will soon see that $f/\psi_Z$ is the function $F$ we are after.

We first show the equivalence between (a) and (c). The calculations above together with Proposition~\ref{prop:zeroseq} show that if $f \in A^p_\om$ and $Z = \Zcal(f)$, then $\sum_{a \in Z} (1-|a|)^2 < \infty$ and $FW_Z = C^{-1}h \in L^p_\om$, where $F = f/\psi_Z$ has no zeros. Conversely, if $F$ is a nowhere zero analytic function with $FW_Z \in L^p_\om$, $\sum_{a \in Z} (1-|a|)^2 < \infty$, then $f = F\psi_Z \in A^p_\om$ because $|\psi_Z(z)| \leq W_Z(z)$ for all $z$, and clearly $\Zcal(f) = Z$.

Since considering a subsequence $Z'$ of $Z$ instead of $Z$ will only decrease the values of $k_Z$ and $W_Z$, it is clear that (a) and (b) are equivalent; (a) $\Rightarrow$ (c) $\Rightarrow$ (b) $\Rightarrow$ (a). Now, if $F$ is any (nonzero) analytic function with $FW_{Z'} \in L^p_\om$, $\sum_{a \in Z'} (1-|a|)^2 < \infty$, then $f = F\psi_{Z'} \in A^p_\om$ and $Z' \subset \Zcal(f)$. The equivalence of (a) and (b) now gives $Z' \in \Zcal(A^p_\om)$, and thus (d) implies (a). Since (c) is a special case of (d), the equivalence part of the theorem is proved.

Finally, because of the part of the proof considering the function $h$, it is clear that the last statement of the theorem is simply a restatement of Proposition~\ref{prop:zeroseq}.
\end{proofof}

\begin{proofof}{Theorem~\ref{theo:factorization}}
The first and the last inequality are obvious, so it suffices to prove the middle one. For $0<p<q<\infty$, $\om\in B_\infty$ and $f\in A^p_\om$, consider the function
    $$
    g(z)=|f(z)|^p\prod_{z_k\in\mathcal{Z}(f)}\frac{1-\frac{p}{q}+\frac{p}{q}|\vp_{z_k}(z)|^q}{|\vp_{z_k}(z)|^p},\quad z\in\D,
    $$
defined in \cite[Lemma~2]{HorFacto}, and let $h$ be as in Proposition~\ref{prop:zeroseq} with $Z=\mathcal{Z}(f)$. Let $n(r,f)$ denote the number of zeros of $f$ in $D(0,r)$, counted according to multiplicity, and
    $$
    N(r,f)=\int_0^r\frac{n(s)-n(0,f)}{s}\,ds+n(0,f)\log r
    $$
be the integrated counting function. Two integrations by parts and Jensen's formula show that
    \begin{equation*}
    \begin{split}\sum_{z_k\in\mathcal{Z}(f)}\log\left(\frac{1-\frac{p}{q}+\frac{p}{q}|z_k|^q}{|z_k|^p}\right)
    &=\int_0^1\log\left(\frac{1-\frac{p}{q}+\frac{p}{q}r^q}{r^p}\right)\,dn(r)
    =\int_0^1\frac{(\frac{q}{p}-1)(1-r^q)}{\frac
    qp-1+r^q}\frac{pn(r)}{r}\,dr\\
    &=-\int_0^1pN(r)\,du(r)
    =\int_\D\log|f(w)|^p\,d\sigma(w)-\log|f(0)|^p,
    \end{split}
    \end{equation*}
where
    $$
    d\sigma(w)=-u'(|w|)\frac{dA(w)}{2|w|},\quad u(r)=\frac{(\frac{q}{p}-1)(1-r^q)}{\frac{q}{p}-1+r^q},
    $$
is a positive measure of unit mass on $\D$. Hence
    \begin{equation*}
    \begin{split}
    \log(g(0))&=\log\left(|f(0)|^p\prod_{z_k\in\mathcal{Z}(f)}\left(\frac{1-\frac{p}{q}+\frac{p}{q}|z_k|^q}{|z_k|^p}\right)\right)\\
    &=\log|f(0)|^p+\sum_{z_k\in\mathcal{Z}(f)}\log\left(\frac{1-\frac{p}{q}+\frac{p}{q}|z_k|^q}{|z_k|^p}\right)
    =\int_\D\log|f(w)|^p\,d\sigma(w).
    \end{split}
    \end{equation*}
Replacing now $f$ by $f\circ\vp_z$ and using \cite[(3.9)]{PR2014Memoirs} to pass from $d\s$ to $dA$, we obtain
    \begin{equation*}
    \begin{split}
    \log(g(z))
    &=\int_\D\log|f(\vp_z(w))|^p\,d\sigma(w)
    \leq \int_\D \log |f(\vp_z(w))|^p\,dA(w) \\
    &=p\int_\D \log |f(\vp_z(w))|\,dA(w) = p\log h(z) = \log h(z)^p, \quad z \notin\Zcal(f).
    \end{split}
    \end{equation*}
Thus $g \leq h^p$, from which Proposition~\ref{prop:zeroseq} gives $\|g\|_{L^1_\om}\le C\|f\|_{A^p_\om}^p$ for some constant $C=C(\om)>0$. This is the statement of \cite[Lemma~3.3]{PR2014Memoirs} with the difference that now we have better control over the constant appearing on the right-hand side of the inequality and $\om$ is only required to belong to~$B_\infty$. By following the proof of \cite[Theorem~3.1]{PR2014Memoirs}, which in turn follows Horowitz' original probabilistic argument, now gives the assertion of Theorem~\ref{theo:factorization} for functions $f$ with finitely many zeros. To complete the argument used in the said proof, it suffices to show that every norm-bounded family in $A^p_\om$ with $\om\in B_\infty$ is a normal family of analytic functions. To see this, let $x>1$ such that $\om\in B_x$. Then the subharmonicity and H\"older's inequality yield
    \begin{equation*}
    \begin{split}
    |f(z)|^\frac{p}{x}&\lesssim\frac{1}{(1-|z|)^2}\int_{\Delta(z,r)}|f(\z)|^\frac{p}{x}\om(\z)^\frac1x\om(\z)^{-\frac1x}\,dA(\z)\\
    &\le\frac{1}{(1-|z|)^2}\left(\int_{\Delta(z,r)}|f(\z)|^p\om(\z)\,dA(\z)\right)^\frac1x
    \left(\int_{\Delta(z,r)}\om(\z)^{-\frac{x'}x}\,dA(\z)\right)^\frac1{x'}\\
    &\le\frac{1}{(1-|z|)^2}\|f\|_{A^p_\om}^\frac{p}x\left(\int_{\D}\om(\z)^{-\frac{x'}x}\,dA(\z)\right)^\frac1{x'},
    \end{split}
    \end{equation*}
where the constant of comparison depends only on the fixed $r\in(0,1)$. Since $\om\in B_x$, the last integral is finite, and Montel's theorem shows that every norm-bounded family in $A^p_\om$ is a normal family of analytic functions. With this guidance we consider Theorem~\ref{theo:factorization} proved.
\end{proofof}

\section{Zeros and factorization when $\om\in\DD$}\label{Sec:radial-zeros-factorization}

We will need the following technical auxiliary result~\cite[Lemma~1]{PR2015Embedding}.

\begin{letterlemma}\label{Lemma:weights-in-D-hat}
Let $\om$ be a radial weight. Then the following statements are equivalent:
\begin{itemize}
\item[\rm(i)] $\om\in\DD$;
\item[\rm(ii)] There exist $C=C(\om)>0$ and $\b=\b(\om)>0$ such that
    \begin{equation*}
    \begin{split}
    \widehat{\om}(r)\le C\left(\frac{1-r}{1-t}\right)^{\b}\widehat{\om}(t),\quad 0\le r\le t<1;
    \end{split}
    \end{equation*}
\item[\rm(iii)] There exist $C=C(\om)>0$ and $\gamma=\gamma(\om)>0$ such that
    \begin{equation*}
    \begin{split}
    \int_0^t\left(\frac{1-t}{1-s}\right)^\g\om(s)\,ds
    \le C\widehat{\om}(t),\quad 0\le t<1.
    \end{split}
    \end{equation*}
\end{itemize}
\end{letterlemma}

\medskip

\begin{proofof}{Theorem~\ref{Theorem:D-hat}} As explained in the introduction we begin with modifying the proof of Proposition~\ref{prop:zeroseq} so that it covers the case $\om\in\DD$. This boils down to showing \eqref{eq:step-radial-case} for sufficiently large $q=q(\om)>\max\{p,1\}$. To prove \eqref{eq:step-radial-case}, let $q\ge2$ and $k(\z)=(1-|\z|)^\e$, where $\e<1-1/q$ is fixed. Writing $I(f)$ for the left-hand side of \eqref{eq:step-radial-case} and using H\"older's inequality and Fubini's theorem, we have
    \begin{equation*}
    \begin{split}
    I(f)&=\int_\D\left(\int_\D |f(\z)|^\frac{p}{q}\frac{k(\z)}{|1-\overline{z}\z|^2}\frac{k^{-1}(\z)}{|1-\overline{z}\z|^2}\,dA(\z)\right)^q\om_{[2q]}(z)\,dA(z)\\
    &\le\int_\D|f(\z)|^pk(\z)^q\left(\int_\D\frac{\om_{[2q]}(z)}{|1-\overline{z}\z|^{2q}}
    \left(\int_\D\frac{dA(u)}{k(u)^{q'}|1-\overline{z}u|^{2q'}}\right)^{q-1}dA(z)\right)dA(\z).
    \end{split}
    \end{equation*}
Since $\e<1-1/q$,
    \begin{equation*}
    \begin{split}
    \int_\D\frac{\om_{[2q]}(z)}{|1-\overline{z}\z|^{2q}}
    \left(\int_\D\frac{dA(u)}{k(u)^{q'}|1-\overline{z}u|^{2q'}}\right)^{q-1}dA(z)
    &\asymp\int_\D\frac{\om_{[2q-2-\e q]}(z)}{|1-\overline{z}\z|^{2q}}\,dA(z)\\
    &\asymp\int_0^1\frac{\om(s)(1-s)^{2q-2-\e q}}{(1-|\z|s)^{2q-1}}\,ds,
    \end{split}
    \end{equation*}
and therefore $I(f)\lesssim\|f\|_{L^p_\mu}^p$, where
     \begin{equation*}
    \begin{split}
    d\mu(\z)=(1-|\z|)^{\e q}\left(\int_0^{|\z|}\frac{\om(s)}{(1-s)^{1+\e q}}\,ds
    +(1-|\z|)^{1-2q}\int_{|\z|}^1\om_{[2q-2-\e q]}(s)\,ds\right)dA(\z).
    \end{split}
    \end{equation*}
To establish \eqref{eq:step-radial-case}, it now suffices to show that $\mu$ is a $p$-Carleson measure for $A^p_\om$, that is, $\mu(S)\lesssim\om(S)$ for all Carleson squares $S$ by \cite[Theorem~1]{PR2015Embedding}. By Fubini's theorem and the inequality $\e q-2q+1<-q<-1$, the term corresponding to the second summand satisfies
    \begin{equation*}
    \begin{split}
    \int_r^1(1-t)^{\e q+1-2q}\int_{t}^1\om_{[2q-2-\e q]}(s)\,ds\,dt
    \le\int_r^1\om_{[2q-2-\e q]}(s)\int_0^s(1-t)^{\e q-2q+1}\,dt\,ds
    \lesssim\widehat{\om}(r)
    \end{split}
    \end{equation*}
while a similar reasoning (divide the integral from $0$ to $t$ at $r$) for the first term shows that
     \begin{equation*}
    \begin{split}
    \int_r^1(1-t)^{\e q}\left(\int_0^{t}\frac{\om(s)}{(1-s)^{1+\e q}}\,ds\right)dt
    \lesssim(1-r)^{\e q+1}\int_0^{r}\frac{\om(s)}{(1-s)^{1+\e q}}\,ds+\widehat{\om}(r).
    \end{split}
    \end{equation*}
By choosing $\e q$ sufficiently large, Lemma~\ref{Lemma:weights-in-D-hat}(iii) gives what we want. Therefore we have proved the statement of Proposition~\ref{prop:zeroseq} for $\om\in\DD$. The proofs of Theorems~\ref{theo:zeroset} and \ref{theo:factorization} now work as such and thus the statements of these results as well as Corollaries~\ref{theo:zeroset2} and~\ref{coro:zeroPerturb} are valid for $\om\in\DD$.
\end{proofof}

\medskip

\begin{proofof}{Corollary~\ref{corollary:Hankel}} For each $p>1$ and $\om\in\DDD$, the dual of $A^p_\om$ can be identified with $A^{p'}_\om$ via the $A^2_\om$-pairing $\langle f,g\rangle_{A^2_\om}=\lim_{r\to1^-}\int_\D f_r\overline{g_r}\om\,dA$, where $f_r(z)=f(rz)$, by the proofs of \cite[Theorem~6 and Corollary~7]{PR2016TwoWeight}. Therefore the boundedness of $h^\om_{\overline{f}}:A^p_\om\to\overline{A^q_\om}$ is equivalent to
    \begin{equation}\label{Equation:hankel}
    \begin{split}
    |\langle h^\om_{\overline{f}}(g),h\rangle_{A^2_\om}|
    &=\left|\lim_{r\to1^-}\int_\D \overline{f(z)}g(z)\overline{h(r^2z)}\om(z)\,dA(z)\right|\\
    &\lesssim\|g\|_{A^p_\om}\|h\|_{L^{q'}_\om},\quad g\in A^p_\om,\quad h\in \overline{A^{q'}_\om}.
    \end{split}
    \end{equation}
If $f\in A^s_\om$ with $\frac1s=\frac1q-\frac1p$, then H\"older's inequality implies $|\langle h^\om_{\overline{f}}(g),h\rangle_{A^2_\om}|\le\|g\|_{A^p_\om}\|h\|_{L^{q'}_\om}\|f\|_{A^s_\om}$, and hence
$h^\om_{\overline{f}}:A^p_\om\to\overline{A^q_\om}$ is bounded and $\|h^\om_{\overline{f}}\|_{A^p_\om\to\overline{A^q_\om}}\le\|f\|_{A^s_\om}$.

Conversely, assume that $h^\om_{\overline{f}}:A^p_\om\to\overline{A^q_\om}$ is bounded. By Theorem~\ref{Theorem:D-hat} each $k\in A^{s'}_\om$ can be factorized as $k=g\overline{h}$, where $g\in A^p_\om$ and $\overline{h}\in A^{q'}_\om$ with $\|g\|_{A^p_\om}\|\overline{h}\|_{A^{q'}_\om}\asymp\|k\|_{A^{s'}_\om}$. Therefore \eqref{Equation:hankel} implies $|\langle f,k\rangle_{A^2_\om}|\lesssim\|k\|_{A^{s'}_\om}$ for all $k\in A^{s'}_\om$. The duality $(A^s_\om)^*\simeq A^{s'}_\om$ now yields $f\in A^s_\om$. 
\end{proofof}

\section{Dominating sets}

In this section we prove our results on dominating sets. We begin with Theorems~\ref{theo:dominating-set-for-f} and~\ref{theo:dominating-set-for-f-2} concerning individual functions.

\medskip

\begin{proofof}{Theorem~\ref{theo:dominating-set-for-f}} By \cite[Lemma~8]{PRS2015Berezin}, there exists $r=r(\om)\in(0,1)$ such that $|B^\om_z(\z)|\asymp B^\om_z(z)$ for all $\z\in\Delta(z,r)$ and $z\in\D$. Further, $\|B^\om_z\|_{A^2_\om}^2\asymp(\omh(z)(1-|z|))^{-1}$ by \cite[Theorem~1]{PR2016TwoWeight}, and $\omh$ is essentially constant in each hyperbolically bounded set by Lemma~\ref{Lemma:weights-in-D-hat}(ii). By using these facts and the subharmonicity of $|f|^q$  we deduce
    \begin{equation*}
    |f(z)|^q\lesssim\int_{\Delta(z,r)}|f(\z)|^q\frac{|B^\om_z(\z)|^2}{\|B^\om_z\|_{A^2_\om}^2}\frac{\omh(\z)}{1-|\z|}\,dA(\z), \quad z \in \D,
    \end{equation*}
for all $f\in\Hcal(\D)$. Therefore~\eqref{eq:PointEval} is proved once we have shown that there exists a constant $C=C(q,r,\om)>0$ such that
    \begin{equation}\label{eq:carleson-application}
    \int_{\Delta(z,r)}|f(\z)|^q|B^\om_z(\z)|^2\frac{\omh(\z)}{1-|\z|}\,dA(\z)
    \le C\int_\D |f(\z)|^q|B^\om_z(\z)|^2\om(\z)\,dA(\z), \quad z \in \D,
    \end{equation}
for all $f \in A^q_\om$. This looks like a consequence of a Carleson embedding theorem, but these theorems do not seem efficiently applicable as such since $|f|^q|B^\om_z|^2$ can not be written as a power of the modulus of a single analytic function because both $f$ and $B^\om_z$ may have zeros. However, a careful inspection of the proof of \cite[Theorem~3.3]{PelSum14} shows that the argument, and in particular the proof of \cite[Lemma~3.2]{PelSum14}, carries over if any positive power of the modulus of the function involved is subharmonic. Now that $|f|^q|B^\om_z|^2$ has this local property because both $f$ and $B^\om_z$ are analytic, we deduce \eqref{eq:carleson-application} if the measure $\mu_z$ defined by $d\mu_z(\z) = \frac{\omh(\z)}{1-|\z|}\chi_{\Delta(z,r)}(\z)\,dA(\z)$ is a $1$-Carleson measure for $A^1_\om$ with $\|M_\om(\mu_z)\|_{L^\infty}$ uniformly bounded in~$z$, that is, if $\mu_z(S(a)) \lesssim \om(S(a))$ for all $a \in \D\setminus\{0\}$ and $z \in \D$. But clearly, $\Delta(z,r)\subset\{se^{i\t}:\r\le s\le\r+x(1-\r),\,\,|\t-\vp|\le x(1-\r)\}$ for some $x=x(r)\in(0,1)$, $\r=\r(z,r)\in(0,1)$ and $\vp=\vp(z)$. If now $|a|\ge \r+x(1-\r)$ there is nothing to prove, while for otherwise
    \begin{equation*}
    \begin{split}
    \int_{S(a)\cap\Delta(z,r)}\frac{\omh(\z)}{1-|\z|}\,dA(\z)
    &\le(1-|a|)\int_{\max\{|a|,\r\}}^{\r+x(1-\r)}\frac{\omh(s)}{1-s}\,ds\\
    &\le(1-|a|)\widehat{\om}(a)\frac{\r+x(1-\r)-\max\{|a|,\r\}}{1-\r-x(1-\r)}\\
    &\le(1-|a|)\widehat{\om}(a)\frac{x}{1-x}.
    \end{split}
    \end{equation*}
It follows that $\mu_z(S(a)) \lesssim \om(S(a))$ for all $a \in \D\setminus\{0\}$ and $z \in \D$, and thus \eqref{eq:PointEval} is proved.

Let now $f\in A^p_\om$. Then~\eqref{eq:BadPoints} implies
    $$
    |f(z)|^q\chi_E(z) \le \e \int_\D |f(\z)|^q K^\om_z(\z)\om(\z)\,dA(\z),\quad z\in\D.
    $$
Raising this to power $p/q$ and integrating with respect to $\om\,dA$ now gives
    $$
    \int_E |f(z)|^p\om(z)\,dA(z)
    \le \e^{p/q} \int_\D \left(\int_\D |f(\z)|^q K^\om_z(\z)\om(\z)\,dA(\z)\right)^{p/q}\om(z)\,dA(z)=\e^{p/q}I(f).
    $$
Therefore~\eqref{eq:BadPointsEstimate} will be proved if we can show that $I(f)$ is dominated by a constant times $\|f\|_{A^p_\om}^p$. If $\om\in\DDD$, there is an easy to way to deduce this because the maximal Bergman projection $P^+_\om:L^p_\om\to L^p_\om$ is bounded for each $1<p<\infty$ by the proof of \cite[Theorem~3]{PR2016TwoWeight}. To see how~\eqref{eq:BadPointsEstimate} is obtained from this fact, denote $s = p/q$ and let $1 < x < s$ and $g \in L^s_\om$. H\"older's inequality then gives
    \begin{equation*}
    \begin{split}
	\int_\D &\left|\int_\D g(\xi) K^\om_z(\z)\om(\z)\,dA(\z)\right|^s\om(z)\,dA(z) \\
    &\le \int_\D \left(\int_\D |g(\z)|^x |B_z^\om(\z)| \om(\z)\,dA(\z)\right)^{s/x} \left(\int_\D |B_z^\om(\z)|^{\frac{2x-1}{x-1}}\om(\z)\,dA(\z)\right)^{s\frac{x-1}x} \frac{\om(z)}{\|B_z^\om\|^{2s}_{A^2_\om}}\,dA(z) \\
		&\leq \int_\D \left(\int_\D |g(\z)|^x |B_z^\om(\z)| \om(\z)\,dA(\z)\right)^{s/x}\om(z)\,dA(z)\, \sup_{z\in\D} \frac{\|B_z^\om\|^{s\frac{2x-1}x}_{A^y_\om}}{\|B_z^\om\|^{2s}_{A^2_\om}},
\end{split}
\end{equation*}
where $y = \frac{2x-1}{x-1}$. Since $s/x > 1$ and $|g|^x \in L^{s/x}_\om$,~the maximal Bergman projection $P^+_\om:L^\frac{s}{x}_\om\to L^\frac{s}{x}_\om$ is bounded, and hence
	$$
	\int_\D \left(\int_\D |g(\xi)|^x |B_z^\om(\xi)| \om(\xi)\,dA(\xi)\right)^{s/x}\om(z)\,dA(z) \lesssim \|g\|^s_{L^s_\om}.
	$$
On the other hand,~\cite[Theorem~1]{PR2016TwoWeight} gives
	$$
	\frac{\|B_z^\om\|^{s\frac{2x-1}x}_{A^y_\om}}{\|B_z^\om\|^{2s}_{A^2_\om}}
		\asymp \frac{\omh(z)^s(1-|z|)^s}{\left(\omh(z)^{y-1}(1-|z|)^{y-1}\right)^{s\frac{2x-1}{xy}}} = 1,\quad |z| \to 1^-,
	$$
because
	$$
	(y-1)s\frac{2x-1}{xy} = \frac{2x-1-(x-1)}{x-1}s\frac{2x-1}{x\frac{2x-1}{x-1}} = \frac{x}{x-1}s\frac{x-1}{x} = s.
	$$
This proves~\eqref{eq:BadPointsEstimate} for $\om\in\DDD\subsetneq\DD$. The proof for the whole class $\DD$ is more laborious but it also relies strongly on the kernel estimates used in the above reasoning. We argue as follows. Let $k(\z)=\widehat{\om}(\z)^\e$, where $\e<1-\frac{q}{p}$. Since $\|B^\om_z\|_{A^2_\om}^2\asymp(\widehat{\om}(z)(1-|z|))^{-1}$, H\"older's inequality and Fubini's theorem yield
    \begin{equation*}
    \begin{split}
    I(f)
    &\asymp\int_\D \left(\int_\D |f(\z)|^q|B^\om_z(\z)|k(\z)|B^\om_z(\z)|k(\z)^{-1}\om(\z)\,dA(\z)\right)^{p/q}(\widehat{\om}(z)(1-|z|))^{\frac{p}{q}}\om(z)\,dA(z)\\
    &\le\int_\D\left(\int_\D |f(\z)|^p|B^\om_z(\z)|^\frac{p}{q}k(\z)^\frac{p}{q}\om(\z)\,dA(\z)\right)\\
    &\quad\cdot\left(\int_\D|B^\om_z(u)|^{\frac{p}{p-q}}k(u)^{-\frac{p}{p-q}}\om(u)\,dA(u)\right)^{\frac{p-q}{q}}(\widehat{\om}(z)(1-|z|))^{\frac{p}{q}}\om(z)\,dA(z)\\
    &=\int_\D|f(\z)|^pk(\z)^\frac{p}{q}\om(\z)\left(\int_\D |B^\om_z(\z)|^\frac{p}{q}(\widehat{\om}(z)(1-|z|))^{\frac{p}{q}}\om(z)\right.\\
    &\quad\cdot\left.\left(\int_\D|B^\om_z(u)|^{\frac{p}{p-q}}k(u)^{-\frac{p}{p-q}}\om(u)\,dA(u)\right)^{\frac{p-q}{q}}\,dA(z)\right)\,dA(\z),
    \end{split}
    \end{equation*}
where
    \begin{equation*}
    \begin{split}
    \int_\D|B^\om_z(u)|^{\frac{p}{p-q}}k(u)^{-\frac{p}{p-q}}\om(u)\,dA(u)
    &\lesssim\int_0^{|z|}\frac{\int_t^1\widehat{\om}(s)^{-\frac{\e p}{p-q}}\om(s)\,ds}{\widehat{\om}(t)^{\frac{p}{p-q}}(1-t)^\frac{p}{p-q}}\,dt\\
    &\asymp\int_0^{|z|}\frac{dt}{\widehat{\om}(t)^{\frac{p}{p-q}-1+\frac{\e p}{p-q}}(1-t)^\frac{p}{p-q}}\\
    &\asymp\frac{1}{\widehat{\om}(z)^{\frac{p}{p-q}-1+\frac{\e p}{p-q}}(1-|z|)^{\frac{p}{p-q}-1}}
    \end{split}
    \end{equation*}
by~\cite[Theorem~1]{PR2016TwoWeight}. This together with another application of ~\cite[Theorem~1]{PR2016TwoWeight} gives
    \begin{equation*}
    \begin{split}
    I(f)
    &\lesssim\int_\D|f(\z)|^p\widehat{\om}(\z)^\frac{\e p}{q}\om(\z)\left(\int_\D |B^\om_z(\z)|^\frac{p}{q}\widehat{\om}(z)^{\frac{p}{q}-1-\frac{\e p}{q}}\om(z)(1-|z|)^{\frac{p-q}{q}}\,dA(z)\right)dA(\z)\\
    &\lesssim\int_\D|f(\z)|^p\widehat{\om}(\z)^\frac{\e p}{q}\om(\z)\left(\int_0^{|\z|} \frac{\int_t^1\widehat{\om}(s)^{\frac{p}{q}-1-\frac{\e p}{q}}\om(s)(1-s)^{\frac{p-q}{q}}\,ds}{\widehat{\om}(t)^\frac{p}{q}(1-t)^\frac{p}{q}}\,dt\right)dA(\z).
    \end{split}
    \end{equation*}
The integral from $0$ to $|\z|$ equals to $I_1+I_2$, where
    \begin{equation*}
    \begin{split}
    I_1&=\left(\int_{|\z|}^1\widehat{\om}(s)^{\frac{p}{q}-1-\frac{\e p}{q}}\om(s)(1-s)^{\frac{p-q}{q}}\,ds\right)
    \int_0^{|\z|} \frac{dt}{\widehat{\om}(t)^\frac{p}{q}(1-t)^\frac{p}{q}}\\
    &\lesssim\left((1-|\z|)^\frac{p-q}{q}\widehat{\om}(\z)^{\frac{p}{q}-\frac{\e p}{q}}\right)
    \frac{1}{\widehat{\om}(\z)^\frac{p}{q}(1-|\z|)^{\frac{p}{q}-1}}=\widehat{\om}(\z)^{-\frac{\e p}{q}}
    \end{split}
    \end{equation*}
and, by Fubini's theorem,
    \begin{equation*}
    \begin{split}
    I_2&=\int_0^{|\z|} \frac{\int_{t}^{|\z|}\widehat{\om}(s)^{\frac{p}{q}-1-\frac{\e p}{q}}\om(s)(1-s)^{\frac{p-q}{q}}\,ds}{\widehat{\om}(t)^\frac{p}{q}(1-t)^\frac{p}{q}}\,dt\\
    &=\int_0^{|\z|}\widehat{\om}(s)^{\frac{p}{q}-1-\frac{\e p}{q}}\om(s)(1-s)^{\frac{p-q}{q}}
    \left(\int_{0}^{s}\frac{dt}{\widehat{\om}(t)^\frac{p}{q}(1-t)^\frac{p}{q}}\right)ds\\
    &\asymp\int_0^{|\z|}\widehat{\om}(s)^{-1-\frac{\e p}{q}}\om(s)\,ds\asymp\widehat{\om}(\z)^{-\frac{\e p}{q}}.
    \end{split}
    \end{equation*}
We deduce $I(f)\lesssim\|f\|_{A^p_\om}^p$, and the theorem is proved.
\end{proofof}

\medskip

\begin{proofof}{Theorem~\ref{theo:dominating-set-for-f-2}}
By the definition of the set $E$ and Fubini's theorem,
    \begin{equation*}
    \begin{split}
    \int_E |f(z)|^p\om(z)\,dA(z)
    &\le \e \int_E \left(\frac1{\nu(E(z))}\int_{E(z)}|f(\zeta)|^p\nu(\zeta)\,dA(\zeta)\right)\om(z)\,dA(z)\\
    &=\e \int_\D |f(\zeta)|^p\left(\int_E \frac{\chi_{E(z)}(\zeta)}{\nu(E(z))} \om(z)\,dA(z)\right)\nu(\zeta)\,dA(\zeta).
    \end{split}
    \end{equation*}
Therefore to complete the proof, it suffices to show that the measure $\mu$ defined by
    $$
    d\mu(\z) =\left(\int_E \frac{\chi_{E(z)}(\zeta)}{\nu(E(z))} \om(z)\,dA(z)\right)\nu(\zeta)\,dA(\zeta),\quad \z\in\D,
    $$
is a $p$-Carleson measure for $A^p_\om$. To see this, let $a\in\D\setminus\{0\}$. Then Fubini's theorem gives
    \begin{equation*}
    \begin{split}
    \mu(S(a))
    &=\int_{S(a)}\left(\int_E \frac{\chi_{E(z)}(\zeta)}{\nu(E(z))} \om(z)\,dA(z)\right)\nu(\zeta)\,dA(\zeta)
    =\int_E\frac{\nu(S(a) \cap E(z))}{\nu(E(z))} \om(z)\,dA(z)\\
    &\le\int_{\{z:S(a)\cap \Delta(z,r)\ne\emptyset\}}\om(z)\,dA(z).
    \end{split}
    \end{equation*}
Now that there exist $C=C(r)>0$ and $R=R(r,C)\in(0,1)$ such that $\{z:S(a)\cap \Delta(z,r)\ne\emptyset\}\subset S(b)$ for some $b=b(a)\in\D\setminus\{0\}$ with $\arg b=\arg a$ and $1-|b|=C(1-|a|)$ and for all $|a|\ge R$, we deduce
    \begin{equation*}
    \begin{split}
    \mu(S(a))
    &\le\om(S(b)) \le C(1-|a|) \int_{1-C(1-|a|)}^1\om(s)\,ds\\
    &\lesssim C(1-|a|)\left(\frac{C(1-|a|)}{1-|a|}\right)^\b\widehat{\om}(a)\asymp
    \om(S(a)), \quad |a|\ge R,
    \end{split}
    \end{equation*}
by the hypothesis $\om\in\DD$ and Lemma~\ref{Lemma:weights-in-D-hat}. Consequently, $\mu$ is a $p$-Carleson measure for $A^p_\om$ by \cite[Theorem~1]{PR2015Embedding}, and the lemma is proved.
\end{proofof}

We next proceed towards the proof of Theorem~\ref{theo:dominating-sets-non-radial}.

\begin{lem}\label{Lemma:test-functions-non-radial}
Let $\om$ be a weight on $\D$ such that $\om(S(a))>0$ for all $a\in\D\setminus\{0\}$. Then the following statements are equivalent:
    \begin{itemize}
    \item[\rm(i)] $\om\in\DD(\D)$;
    \item[\rm(ii)] there exist $\b=\b(\om)>0$ and $C=C(\om)\ge1$ such that
        $$
        \frac{\om(S(a))}{(1-|a|)^\b}\le C\frac{\om(S(a'))}{(1-|a'|)^\b},\quad 0<|a|\le|a'|<1,\quad\arg a=\arg a';
        $$
    \item[\rm(iii)] for some (equivalently for each) $K>0$ there exists $C=C(\om,K)>0$ such that
        $$
        \om(S(a))\le C\om\left(S\left(\frac{K+|a|}{K+1}e^{i\arg a}\right)\right),\quad a\in\D\setminus\{0\};
        $$
    \item[\rm(iv)] there exist $\eta=\eta(\om)>0$ and $C=C(\eta,\om)>0$ such that
        $$
        \displaystyle
        \int_\D\frac{\om(z)}{|1-\overline{a}z|^\eta}\,dA(z)\le C\frac{\om(S(a))}{(1-|a|)^\eta},\quad a\in\D\setminus\{0\}.
        $$
    \end{itemize}
\end{lem}

\begin{proof}
Assume (i), and let $r_n=1-2^{-n}$ for $n\in\N\cup\{0\}$. Let $0<|a|\le|a'|<1$ with $\arg a=\arg a'$, and take $k,m\in\N\cup\{0\}$ such that $r_k\le|a|<r_{k+1}$ and $r_m\le|a'|<r_{m+1}$. Then
    \begin{equation*}
    \begin{split}
    \om(S(a))
    &\le\om\left(S\left(r_k\frac{a}{|a|}\right)\right)
    \le C\om\left(S\left(r_{k+1}\frac{a}{|a|}\right)\right)
    \le\cdots\le C^{m-k+1}\om\left(S\left(|a'|\frac{a}{|a|}\right)\right)\\
    &=C^{2}2^{(m-k-1)\log_2 C}\om\left(S\left(a'\right)\right)
    \le C^2\left(\frac{1-|a|}{1-|a'|}\right)^{\log_2C}\om\left(S\left(a'\right)\right),
    \end{split}
    \end{equation*}
and thus (ii) is proved.

Assume (ii) and let $K>0$. For each $a\in\D\setminus\{0\}$ take $a'=\frac{K+|a|}{K+1}e^{i\arg a}$. Then (ii) implies
    \begin{equation*}
    \begin{split}
    \frac{\om(S(a))}{(1-|a|)^\b}\le C(K+1)^\beta\frac{\om(S(a'))}{(1-|a|)^\b},\quad a\in\D\setminus\{0\},
    \end{split}
    \end{equation*}
and thus it follows that (iii) is satisfied for each $K>0$.

Assume now (iii) for some $K>0$, and let $a_n=(1-(K+1)^n(1-|a|))e^{i\arg a}$ for $n=0,\ldots,N=\max\{k\in\N\cup\{0\}:(K+1)^k(1-|a|)<1\}$. Let $\eta>0$ to be fixed later. If $z\in S(a_n)\setminus S(a_{n-1})$, $n \geq 2$, then
    $$
    |1-\overline{a}z|\ge|a-z|\gtrsim|a-a_{n-1}|=(1-|a|)((K+1)^{n-1}-1)\asymp(1-|a|)(K+1)^n.
    $$
Moreover, since $\frac{K+|a_n|}{K+1}=|a_{n-1}|$, the hypothesis (iii) yields
    \begin{equation*}
    \begin{split}
    \int_{\D\setminus S(a_1)}\frac{\om(z)}{|1-\overline{a}z|^\eta}\,dA(z)
    &=\sum_{n=2}^N\int_{S(a_n)\setminus S(a_{n-1})}\frac{\om(z)}{|1-\overline{a}z|^\eta}\,dA(z)
    +\int_{\D\setminus S(a_N)}\frac{\om(z)}{|1-\overline{a}z|^\eta}\,dA(z)\\
    &\lesssim\frac1{(1-|a|)^\eta}\sum_{n=2}^N\frac{\om(S(a_n))}{(K+1)^{n\eta}}+1
    \le\frac{\om(S(a_0))}{(1-|a|)^\eta}\sum_{n=2}^\infty\left(\frac{C}{(K+1)^{\eta}}\right)^{n}+1.
    \end{split}
    \end{equation*}
In addition, clearly
    $$
    \int_{S(a_1)}\frac{\om(z)}{|1-\overline{a}z|^\eta}\,dA(z)
    \lesssim \frac{\om(S(a_1))}{(1-|a_1|)^\eta}.
    $$
By repeating the method used in the first part of the proof, it is easy to see that (iii) in fact implies (ii). Therefore (iv) follows by choosing $\eta>\log_{K+1}C$ sufficiently large so that $\om(S(a))(1-|a|)^{-\eta}$ is essentially increasing.

Finally, assume (iv) and denote $a^\star=\frac{1+|a|}{2}e^{i\arg a}$ for $a\in\D\setminus\{0\}$. Then
    \begin{equation*}
    \frac{\om(S(a))}{(1-|a^\star|)^\eta}
    \asymp\int_{S(a)}\frac{\om(z)}{|1-\overline{a^\star}z|^\eta}\,dA(z)
    \leq\int_{\D}\frac{\om(z)}{|1-\overline{a^\star}z|^\eta}\,dA(z)
    \lesssim\frac{\om(S(a^\star))}{(1-|a^\star|)^\eta},\quad a\in\D\setminus\{0\},
    \end{equation*}
and it follows that $\om\in\DD(\D)$.
\end{proof}

\begin{proofof}{Theorem~\ref{theo:dominating-sets-non-radial}}
Let $G$ be a dominating set for $A^p_\om$. It suffices to show~\eqref{eq:delta-condition-set-G-non-radial} for points $a \in \D$ close to the boundary. Let $\beta = \beta(\om) > 0$ and $\eta = \eta(\om) > 0$ be those in Lemma~\ref{Lemma:test-functions-non-radial}.
For each $a\in\D\setminus\{0\}$, define
    $$
    \widetilde{f}_a(z) = \left(\frac{1-|a|}{1-\overline{a}z}\right)^{m/p} \frac1{\om(S(a))^{1/p}}, \quad z\in\D,
    $$
where $m = m(\om)> \max\{\beta,\eta\}$. Then $\big\|\widetilde{f}_a\big\|_{A^p_\om}\asymp1$ for all $a\in\D\setminus\{0\}$ by Lemma~\ref{Lemma:test-functions-non-radial}. Set $f_a=\widetilde{f}_a/\big\|\widetilde{f}_a\big\|_{A^p_\om}$ for all $a\in\D\setminus\{0\}$.

For given $a$ close to the boundary, consider the points $a_n$ defined in the proof of Lemma~\ref{Lemma:test-functions-non-radial}. Then the proof of (iii) implies (iv) with $K=1$ in the said lemma gives
    \begin{equation*}
    \begin{split}
    \int_{\D\setminus S(a_n)}|\widetilde{f}_a(z)|^p\om(z)\,dA(z)
    &=\frac{(1-|a|)^m}{\om(S(a))}\int_{\D\setminus S(a_n)}\frac{\om(z)}{|1-\overline{a}z|^m}\,dA(z)\\
    &\lesssim\sum_{j=n+1}^\infty\left(\frac{C}{2^m}\right)^j\asymp\frac{1}{2^{mn}},\quad n\in\N,
    \end{split}
    \end{equation*}
for all $|a|$ sufficiently close to the boundary and $m$ sufficiently large. It follows that
    $$
    \int_{\D\setminus S(a_n)}|f_a(z)|^p\om(z)\,dA(z)\lesssim\frac1{2^{mn}},\quad n\in\N.
    $$
Since $G$ is a dominating set for $A^p_\om$ by the hypothesis, there exists $\d=\d(G)>0$ such that
    \begin{equation}\label{5}
    \d\|f\|^p_{A^p_\om}\le\int_G |f(z)|^p\om(z)\,dA(z), \quad f \in A^p_\om.
    \end{equation}
Fix $n\in\N$ sufficiently large such that
    $$
    \int_{\D\setminus  S(a_n)} |f_a(z)|^p\om(z)\,dA(z) < \frac\delta{2}.
    $$
Since $f_a$ has norm one in $A^p_\om$, this together with \eqref{5} yields
    \begin{equation*}
    \begin{split}
    \int_{G\cap  S(a_n)} |f_a(z)|^p\om(z)\,dA(z)
    &= \int_{G} |f_a(z)|^p\om(z)\,dA(z) - \int_{G\setminus  S(a_n)} |f_a(z)|^p\om(z)\,dA(z) \\
    &\geq \delta - \int_{\D\setminus  S(a_n)} |f_a(z)|^p\om(z)\,dA(z)
    > \d - \frac\d{2} = \frac\d{2},
\end{split}
\end{equation*}
and
    $$
    \int_{G\cap  S(a_n)} |f_a(z)|^p\om(z)\,dA(z)
    \asymp\int_{G\cap  S(a_n)} |\widetilde{f}_a(z)|^p\om(z)\,dA(z)
    \le\frac{\om(G\cap S(a_n))}{\om(S(a))}.
    $$
Since $\om\in\DD(\D)$, we have $\om(S(a))=\om(S(a_0))\ge C^{-n}\om(S(a_n))$, and now that $n$ is fixed, we deduce $\om(S(a_n))\lesssim\om(G\cap S(a_n))$ for all $a\in\D$ sufficiently close to the boundary. The claim \eqref{eq:delta-condition-set-G-non-radial} follows from this estimate.
\end{proofof}

\section{Sampling measures}

A positive Borel measure $\mu$ on $\D$ is a $q$-Carleson measure for $A^p_\om$ if $A^p_\om$ is continuously embedded into $L^q_\mu$.
In order to prove Theorem~\ref{theo:sampling-suff} we need the following lemma which is a generalization of~\cite[Theorem~2.3]{L1985}.
It readily follows from the proof that if the hypothesis on $\nu$ is replaced by $\nu(\Delta(a,r))\lesssim\om(\Delta(a,r))$, then on the left in both occasions $S(\z)$ must be replaced by $\Delta(\z,r)$. Luecking~\cite[Lemma~3.10]{L1985/2} showed this for $\nu=\om\in C_\infty$ under the hypothesis $\mu(\Delta(a,r))\lesssim\om(\Delta(a,r))$. Recall that, as discussed after Theorem~\ref{Theorem:D-hat}, this last requirement characterizes $p$-Carleson measures for $A^p_\om$ if $\om$ satisfies the the Bekoll\'e-Bonami condition by \cite[Theorem~3.1]{Constantin2}.

\begin{lem}\label{lem:difference-estimate}
Let $0 < p < \infty$, $0<r<\frac{R}2\le\frac14$ and $\om$ a weight. Let $\mu$ and $\nu$ be positive Borel measures on $\D$ such that $d\widetilde{\mu}(z) = \frac{\mu(\Delta(z,R))}{(1-|z|)^2}\,dA(z)$ is a $p$-Carleson measure for $A^p_\om$ and $\nu(S(a))\lesssim\om(S(a))$ for all $a\in\D\setminus\{0\}$. Then there exists a constant $C = C(p,R,\om,\mu,\nu) > 0$ such that
    $$
    \int_\D\left(\int_{S(\z)}|f(z)-f(\zeta)|^p\,d\nu(z)\right)\frac{d\mu(\zeta)}{\om(S(\zeta))}
    \le r^pC \|f\|^p_{A^p_\om}, \quad f \in A^p_\om.
    $$
In particular, if $\om\in\DD$ and $\mu$ is a $p$-Carleson measure for $A^p_\om$, then the statement is valid.
\end{lem}

\begin{pro}
First, it is easy to show that there exists a constant $C = C(p,R) > 0$ such that
    $$
    \left|\frac{f(z) - f(0)}{z}\right|^p \le C\int_{D(0,R)} |f(w)|^p\,dA(w)
    $$
for all $|z| < R/2$ and $f \in \mathcal{H}(\D)$. Thus, if $|z| < r < R/2$, then
    $$
    |f(z) - f(0)|^p
    \le r^p C \int_{D(0,R)} |f(w)|^p\,dA(w).
    $$
An application of this to $f\circ\vp_\zeta$ and a change of variable on the right yield
    $$
    |f(z) - f(\zeta)|^p
    \le r^p C \int_{\Delta(\zeta,R)} |f(w)|^p\frac{(1-|\zeta|^2)^2}{|1-\overline{\zeta}w|^4}\,dA(w),\quad z\in\Delta(\z,r).
    $$
By multiplying this by $\chi_{S(\z)}(z)/\om(S(\zeta))$, integrating with respect to $\nu$ in the variable $z$ and using the hypothesis on $\nu$ now give
    \begin{equation*}
    \begin{split}
    \int_{S(\z)}|f(z)-f(\zeta)|^p\frac{d\nu(z)}{\om(S(\zeta))}
    &\le\frac{4Cr^p}{\om(S(\zeta))} \int_{S(\z)}\left(\int_{\Delta(\zeta,R)}|f(w)|^p\frac{dA(w)}{(1-|w|)^{2}}\right)d\nu(z)\\
    &\lesssim r^p \int_{\Delta(\zeta,R)} |f(w)|^p\frac{dA(w)}{(1-|w|)^{2}}.
    \end{split}
    \end{equation*}
Now an integration with respect to $\mu$ in the variable $\zeta$ and Fubini's theorem yield
    \begin{equation*}
    \begin{split}
    \int_\D\left(\int_{S(\z)}|f(z)-f(\zeta)|^p\,d\nu(z)\right)\frac{d\mu(\zeta)}{\om(S(\zeta))}
    &\lesssim r^p\int_\D|f(w)|^p\frac{\mu(\Delta(w,R))}{(1-|w|)^2}\,dA(w)=r^p\|f\|_{L^p_{\widetilde{\mu}}}^p.
    \end{split}
    \end{equation*}
Therefore the first statement in the lemma follows by the hypothesis on $\widetilde\mu$.

It remains to show that if $\om\in\DD$, then $\widetilde{\mu}$ is a $p$-Carleson measure for $A^p_\om$ whenever $\mu$ is. To see this, use first Fubini's theorem to deduce
    \begin{equation*}
    \begin{split}
    \widetilde{\mu}(S(a)) &= \int_{S(a)} \frac{\mu(\Delta(z,R))}{(1-|z|)^2}\,dA(z)
    =\int_{\{\z:S(a)\cap\Delta(\z,R)\ne\emptyset\}}\left(\int_{S(a)\cap\Delta(\z,R)}\frac{dA(z)}{(1-|z|)^2}\right)d\mu(\z)\\
    &\le\int_{S(b)}\left(\int_{\Delta(\z,R)}\frac{dA(z)}{(1-|z|)^2}\right)d\mu(\z)\asymp\mu(S(b)),
    \end{split}
    \end{equation*}
where $b=b(a,R)\in \D$ is such that $\arg b=\arg a$ and $1-|b|\asymp1-|a|$ for all $a\in\D$. Now that $\mu$ is a $p$-Carleson measure for $A^p_\om$ by the hypothesis, \cite[Theorem~1]{PR2015Embedding} together with Lemma~\ref{Lemma:weights-in-D-hat} yields $\widetilde{\mu}(S(a))\lesssim\om(S(b))\asymp\om(S(a))$ for all $a\in\D$. Consequently, $\widetilde{\mu}$ is indeed a $p$-Carleson measure for $A^p_\om$, and the proof is complete.
\end{pro}

\begin{proofof}{Theorem~\ref{theo:sampling-suff}}
By Lemma~\ref{lem:difference-estimate}, with $d\nu = \om\,dA$ and $\Delta(\z,r)$ in place of $S(\z)$, and \cite[Lemma~3.10]{L1985/2} there exists a constant $C=C(p,R,\om,\mu)>0$ such that
    \begin{equation}\label{eq:proof-theo-sampling-suff}
    \int_\D\left(\int_{\Delta(z,r)}|f(z)-f(\zeta)|^p\frac{d\mu(\zeta)}{\om(\Delta(\zeta,r))}\right)\om(z)\,dA(z)
    \leq r^pC \|f\|^p_{A^p_\om}, \quad f \in A^p_\om,
\end{equation}
where $0<r<R/2\le1/4$. Let first $p \geq 1$. Then raising this to power $1/p$ and using Minkowski's inequality on the left yields
    \begin{equation*}
    \begin{split}
    &\Bigg(\int_\D |f(z)|^p\left(\int_{\Delta(z,r)}\frac{d\mu(\z)}{\om(\Delta(\zeta,r))}\right)\om(z)\,dA(z)\Bigg)^{1/p}\\
    &- \left(\int_\D\left(\int_{\Delta(z,r)}|f(\zeta)|^p\frac{d\mu(\zeta)}{\om(\Delta(\zeta,r))}\right)\om(z)\,dA(z)\right)^{1/p}
    \le rC^{1/p} \|f\|_{A^p_\om},
    \end{split}
    \end{equation*}
where, by Lemma~\ref{Lemma:weights-in-D-hat} and \cite[Lemma~3.4]{L1985/2}, the fact that $r$ is bounded away from 1 and the definition of $G$,
    \begin{equation*}
    \int_{\Delta(z,r)} \frac{d\mu(\zeta)}{\om(\Delta(\zeta,r))}
    \ge C_1 k_r(z)
    \ge C_1\e \|M_\om(\mu)\|_{L^\infty}\chi_G(z),\quad z\in\D,
    \end{equation*}
for some constant $C_1=C_1(\om)>0$, and thus
    $$
    \int_\D |f(z)|^p\left(\int_{\Delta(z,r)}\frac{d\mu(\z)}{\om(\Delta(\z,r))}\right)\om(z)\,dA(z)
    \ge C_1\e\|M_\om(\mu)\|_{L^\infty} \int_G |f(z)|^p\om(z)\,dA(z).
    $$
Further, by Fubini's theorem,
    \begin{equation*}
    \begin{split}
    \int_\D\left(\int_{\Delta(z,r)}|f(\zeta)|^p\frac{d\mu(\zeta)}{\om(\Delta(\z,r))}\right)\om(z)\,dA(z)
    &=\int_\D |f(\zeta)|^p\,d\mu(\zeta),
    \end{split}
    \end{equation*}
and hence
    $$
    \left(C_1\e\|M_\om(\mu)\|_{L^\infty} \int_G |f(z)|^p\om(z)\,dA(z)\right)^{1/p}
    - \left(\int_\D |f(\zeta)|^p\,d\mu(\zeta)\right)^{1/p}
    \le rC^{1/p} \|f\|_{A^p_\om}.
    $$
Since by the hypothesis $G$ is dominating set, there exists a constant $\alpha>0$ such that $\int_G |f(z)|^p\om(z)\,dA(z)\ge\alpha\|f\|^p_{A^p_\om}$ for all $f \in A^p_\om$. Consequently, choosing $r$ such that $r^pC < C_1\e\alpha\|M_\om(\mu)\|_{L^\infty}$ yields
\begin{equation}\label{eq:proof-theo-sampling-suff-2}
    \|f\|_{A^p_\om}
    \le\frac{1}{(C_1\e\alpha\|M_\om(\mu)\|_{L^\infty})^{1/p}-rC^{1/p}}
    \left(\int_\D |f(\zeta)|^p\,d\mu(\zeta)\right)^{1/p},
\end{equation}
and thus the proof is complete when $1 \leq p < \infty$.

If $p < 1$, then one can simply apply the inequality $|x - y|^p \geq |x|^p - |y|^p$ to the left hand side of~\eqref{eq:proof-theo-sampling-suff} to obtain
    \begin{equation*}
    \begin{split}
    &\int_\D |f(z)|^p\left(\int_{\Delta(z,r)}\frac{d\mu(\z)}{\om(\Delta(\z,r))}\right)\om(z)\,dA(z)\\
    &- \int_\D\left(\int_{\Delta(z,r)}|f(\zeta)|^p\frac{d\mu(\zeta)}{\om(\Delta(\z,r))}\right)\om(z)\,dA(z)
    \le r^p C \|f\|^p_{A^p_\om}.
    \end{split}
    \end{equation*}
Then the estimates used in the case $p \geq 1$ yield
    $$
    C_1\e\|M_\om(\mu)\|_{L^\infty} \int_G |f(z)|^p\om(z)\,dA(z)
    -\int_\D |f(\zeta)|^p\,d\mu(\zeta)
    \le r^pC \|f\|^p_{A^p_\om},
    $$
from which and estimate similar to~\eqref{eq:proof-theo-sampling-suff-2} follows as above.
\end{proofof}

\begin{proofof}{Theorem~\ref{theo:w-conv-meas}}
The first statement can be established by following the proof of~\cite[Theorem 1]{L2000}. Namely, since $\mu_n(S(a)\cap D(0,r))\le\|M_\om(\mu_n)\|_{L^\infty}\om(\D)$ for each $a\in\D$, the hypothesis $\sup_n\|M_\om(\mu_n)\|_{L^\infty}<\infty$ implies that for each $r\in(0,1)$, the sequence $\left(\mu_n|_{D(0,r)}\right)_{n\in\N}$ is bounded in $C_0(D(0,r))^*$. Hence there is a subsequence that converges in the weak$^*$-topology by the Banach-Alaoglu theorem. By diagonalization we may extract a subsequence $(\mu_{n_j})$ such that, for each $r \in (0,1)$, $(\mu_{n_j}|_{D(0,r)})$ converges to a measure $\mu_r$ supported in $\overline{D(0,r)}$. Since clearly $\mu_s = \mu_r$ on $D(0,r)$ for all $r<s$, we may define a measure $\mu$ as the limit $\lim_{r\to1^-} \mu_r$. If now $h\in C_c(\D)$, then it's support is in some $D(0,r)$, and thus
    $$
    \int_\D h(z)\,d\mu_{n_j}(z) \to \int_\D h(z)\,d\mu(z),\quad j \to \infty.
    $$

To prove \eqref{eq:w-conv-meas}, let $f \in A_\om^p$. For $h\in C_c(\D)$ satisfying $h(z) \leq 1$ for all $z \in \D$ we have
    \begin{equation*}
    \liminf_{n\to\infty} \int_\D |f(z)|^p\,d\mu_n(z)
    \ge\lim_{n\to\infty} \int_\D h(z)|f(z)|^p\,d\mu_n(z)
    =\int_\D h(z)|f(z)|^p\,d\mu(z)
    \end{equation*}
by Fatou's lemma, and hence
    \begin{equation}\label{1}
    \liminf_{n\to\infty} \int_\D |f(z)|^p\,d\mu_n(z) \geq \int_\D |f(z)|^p\,d\mu(z).
    \end{equation}
For the converse inequality, let $\e > 0$ and take $r=r(f)\in (0,1)$ such that
    $$
    \int_{\D\setminus\overline{D(0,r)}} |f(z)|^p\,\om(z)\,dA(z) < \e.
    $$
Let $h\in C_c(\D)$ such that $h(z) \leq 1$ for all $z \in \D$ and $h\equiv1$ on $\overline{D(0,r)}$. Then
    \begin{equation}\label{eq:weak-conv-pro}
    \int_\D |f(z)|^p\,d\mu_n(z)
    \le\int_\D h(z)|f(z)|^p\,d\mu_n(z)
    +\int_{\D\setminus\overline{D(0,r)}} |f(z)|^p\,d\mu_n(z).
    \end{equation}
Let us first handle the right most integral. For a moment, let $\nu$ be a $p$-Carleson measure for $A_\om^p$ and denote $d\nu_r = \chi_{\D\setminus\overline{D(0,r)}}\,d\nu$. Then, by~\cite[Theorem 9]{PR2015Embedding},
    \begin{equation*}
    \begin{split}
    \int_{\D\setminus\overline{D(0,r)}} |f(z)|^p\,d\nu(z)
        &= \int_\D \left|f(z)\chi_{\D\setminus\overline{D(0,r)}}(z)\right|^p\,d\nu_r(z) \\
        &\leq \int_\D N\left(f\chi_{\D\setminus\overline{D(0,r)}}\right)^p(z)\,d\nu_r(z) \\
        &\lesssim \int_\D N\left(f\chi_{\D\setminus\overline{D(0,r)}}\right)^p(z)
            M_\om(\nu_r)(z)\,\om(z)\,dA(z) \\
        &\leq \|M_\om(\nu_r)\|_{L^\infty} \int_{\D\setminus\overline{D(0,r)}} N(f)^p(z)\,\om(z)\,dA(z) \\
        &\asymp \|M_\om(\nu_r)\|_{L^\infty} \int_{\D\setminus\overline{D(0,r)}} |f(z)|^p\,\om(z)\,dA(z),\quad f \in A_\om^p,
    \end{split}
    \end{equation*}
where the last inequality follows from the fact that the classical non-tangential maximal function is a bounded operator from $H^p$ to $L^p$ of the boundary~\cite[Theorem~3.1 on p.~57]{Garnett1981}. Therefore there exists a constant $C=C(p)>0$ such that
\begin{equation*}
    \int_{\D\setminus\overline{D(0,r)}} |f(z)|^p\,d\mu_n(z)
        \leq C \sup_{n}
        \|M_\om(\mu_n)\|_{L^\infty} \int_{\D\setminus\overline{D(0,r)}}
            |f(z)|^p\,\om(z)\,dA(z)
        \leq C\Lambda\e.
\end{equation*}
Thus, by taking the limit superior of~\eqref{eq:weak-conv-pro} we obtain
\begin{equation*}
\begin{split}
    \limsup_{n\to\infty} \int_\D |f(z)|^p\,d\mu_n(z)
        &\le\int_\D h(z)|f(z)|^p\,d\mu(z) + C\Lambda\e
        \le\int_\D |f(z)|^p\,d\mu(z) + C\Lambda\e,
\end{split}
\end{equation*}
and since $\e > 0$ was arbitrary, by combining this with \eqref{1} we deduce \eqref{eq:w-conv-meas}.

Recall that if $\nu$ is a $p$-Carleson measure for $A^p_\om$, then $\|Id\|^p_{A^p_\om \to L^p_\nu}\asymp \|M_\om(\nu)\|_{L^\infty}$ by
\cite[Theorem~3]{PelRatSie2015}, and therefore $\sup_n\|Id\|^p_{A^p_\om \to L^p_{\mu_n}}=\Lambda_1<\infty$. It follows that
    $$
    \int_\D |f(z)|^p\,d\mu_n(z) \le\|Id\|_{A^p_\om \to L^p_{\mu_n}}^p\|f\|^p_{A^p_\om}
    \le\Lambda_1\|f\|^p_{A^p_\om}, \quad f \in A^p_\om,
    $$
for all $n \in \N$. By the identity \eqref{eq:w-conv-meas} just proved, we may pass to the limit to obtain
    $$
    \int_\D |f(z)|^p\,d\mu(z)\le\Lambda_1\|f\|^p_{A^p_\om}, \quad f \in A^p_\om.
    $$
Since this is also true for any subsequence of $(\mu_n)$, we may replace $\Lambda_1$ with $\liminf_{n \to \infty} \|Id\|^p_{A^p_\om \to L^p_{\mu_n}}$. Thus $\mu$ is a $p$-Carleson measure for $A^p_\om$ with $\|Id\|_{A^p_\om \to L^p_\mu}\le\liminf_{n\to\infty}\|Id\|_{A^p_\om\to L^p_{\mu_n}}$ as claimed.

In the case of sampling measures, the lower inequality follows in a manner similar to above. Details of this step are omitted.

\end{proofof}

\end{document}